\documentclass[11pt]{amsart}
\usepackage{a4}
\usepackage{amsfonts}
\usepackage{amsmath}
\usepackage{amssymb}

\usepackage{hyperref}
\usepackage{amscd}
\date{}
\numberwithin{equation}{section}
\swapnumbers
\theoremstyle{plain}
\newtheorem{Lemma}{Lemma}
\numberwithin{Lemma}{section}
\newtheorem{Corollary}[Lemma]{Corollary}
\newtheorem{Theorem}[Lemma]{Theorem}

\theoremstyle{definition}

\newtheorem{Definition}[Lemma]{Definition} 

\theoremstyle{remark}
\newtheorem{Remark}[Lemma]{Remark}

\newcommand\A{\mathbb A}
\newcommand\X{\mathbb X}
\newcommand\C{\mathbb C}

\newcommand\Q{\mathbb Q}
\newcommand\R{\mathbb R}
\newcommand\Z{\mathbb Z}

\renewcommand\Re{\operatorname{Re}}
\renewcommand\Im{\operatorname{Im}}

\newcommand\id{\mathrm{id}}
\newcommand\im{\operatorname{im}}
\newcommand\punkt{\mathord{\,\cdot\,}}
\newcommand\ch{\mathrm{ch}}
\newcommand\tr{\operatorname{tr}}
\newcommand\str{\operatorname{str}}

\newcommand\cho{\mathrm{ch}^{\mathrm o}}
\newcommand\chosl{\widetilde{\ch}{}^{\mathrm o}}
\newcommand\rk{\operatorname{rk}}
\newcommand\Jnull{\,{}^0\!J}

\newcommand\eps{\varepsilon}

\begin{document}

\title[Torsion Invariants for Families]{Torsion Invariants for Families}
\author{Sebastian Goette}
\address{Mathematisches Institut\\
Universit\"at Freiburg\\
Eckerstr.~1\\
79104 Freiburg\\
Germany}
\email{sebastian.goette@math.uni-freiburg.de}

\begin{abstract}
  We give an overview over the higher torsion invariants of Bismut-Lott,
  Igusa-Klein and Dwyer-Weiss-Williams, including some more or less recent
  developments.
\end{abstract}
\keywords{Bismut-Lott torsion, Igusa-Klein torsion, Dwyer-Weiss-Williams torsion, higher analytic torsion, higher Franz-Reidemeister torsion}
\subjclass[2000]{58J52 (57R22 55R40)}
\dedicatory{Dedicated to Jean-Michel Bismut on the occasion
of his 60th birthday}
\thanks{Supported in part by DFG special programme
``Global Differential Geometry''}

\maketitle

The classical Franz-Reidemeister torsion~$\tau_{\mathrm{FR}}$ is an invariant
of manifolds with acyclic unitarily flat vector bundles~\cite{Reide}, \cite{Franz}.
In contrast to most other algebraic-topological invariants known at that time,
it is invariant under homeomorphisms and simple-homotopy equivalences,
but not under general homotopy equivalences.
In particular,
it can distinguish homeomorphism types
of homotopy-equivalent lens spaces.
Hatcher and Wagoner suggested in~\cite{HW} to extend~$\tau_{\mathrm{FR}}$
to families of manifolds~$p\colon E\to B$
using pseudoisotopies and Morse theory.
A construction of such a higher Franz-Reidemeister torsion~$\tau$
was first proposed by John Klein in~\cite{Kphd}
using a variation of Waldhausen's $A$-theory.
Other descriptions of~$\tau$ were later given by Igusa and Klein
in~\cite{IKpict}, \cite{IKfilt}.

In this overview,
we will refer to the construction in~\cite{Ibuch}.
Let~$p\colon E\to B$ be a family of smooth manifolds,
and let~$F\to E$ be a unitarily flat complex vector bundle of rank~$r$
such that the fibrewise cohomology with coefficients in~$F$
forms a unipotent bundle over~$B$.
Using a function~$h\colon E\to\R$
that has only Morse and birth-death singularities along each fibre of~$p$,
and with trivialised fibrewise unstable tangent bundle,
one constructs a homotopy class of maps~$\xi_h(M/B;F)$ from~$B$
to a classifying space~$Wh^h(M_r(\C),U(r))$.
Now,
the higher torsion~$\tau(E/B;F)\in H^{4\bullet}(B;\R)$ is defined
as the pull-back of a certain universal cohomology class~$\tau
\in H^{4\bullet}\bigl(Wh^h(M_r(\C),U(r));\R\bigr)$.

On the other hand,
Ray and Singer defined an analytic torsion~$\mathcal T_{\mathrm{RS}}$
of unitarily flat complex vector bundles on compact manifolds in~\cite{RS}
and conjectured that~$\mathcal T_{\mathrm{RS}}=\tau_{\mathrm{FR}}$.
This conjecture was established independently by Cheeger \cite{Cheeger}
and M\"uller~\cite{Mueller}.
The most general comparison result was given by Bismut and Zhang in~\cite{BZ1}
and~\cite{BZ2}.
In~\cite{Wagoner},
Wagoner predicted the existence of a ``higher analytic torsion''
that detects homotopy classes in the diffeomorphism groups
of smooth closed manifolds.
Such an invariant was defined later by Bismut and Lott in~\cite{BL}.

Kamber and Tondeur constructed characteristic
classes~$\cho(F)\in H^{\mathrm{odd}}(M;\R)$
of flat vector bundles~$F\to M$ in \cite{KT}
that provide obstructions towards finding a parallel metric.
If~$p\colon E\to B$ is a smooth bundle of compact manifolds
and~$F\to E$ is flat,
Bismut and Lott proved a Grothendieck-Riemann-Roch theorem
relating the characteristic
classes of~$F$ to those of the fibrewise cohomology~$H(E/B;F)\to B$.
The higher analytic torsion form~$\mathcal T (T^HE,g^{TX},g^F)$
appears in a refinement of this theorem to the level of differential forms.
Its component 
in degree~$0$ equals the Ray-Singer analytic torsion of the fibres,
and the refined Grothendieck-Riemann-Roch theorem implies a variation formula
for the Ray-Singer torsion that was already discovered in~\cite{BZ1}.

In~\cite{DWW},
Dwyer, Weiss and Williams gave yet another approach to higher torsion.
They defined three generalised Euler characteristics
for bundles~$p\colon E\to B$ of homotopy finitely dominated spaces,
topological mani\-folds, and smooth mani\-folds, respectively,
with values in certain bundles over~$B$.
A flat complex vector bundle~$F\to E$ defines a homotopy
class of maps from~$E$ to the algebraic $K$-theory space~$K(\C)$.
The Euler characteristics above give analogous maps~$B\to K(\C)$
for the fibrewise cohomology~$H(E/B;F)\to B$.
If~$F$ is fibrewise acyclic,
these maps lift to three different generalisations of Reidemeister torsion,
given again as sections in certain bundles over~$B$.
By comparing the three characteristics for smooth manifold bundles,
Dwyer, Weiss and Williams also showed that the Grothendieck-Riemann-Roch
theorem in~\cite{BL} holds already on the level of classifying maps to~$K(\C)$.

Bismut-Lott torsion~$\mathcal T(E/B;F)$
and Igusa-Klein torsion~$\tau(E/B;F)$ are very closely related.
For particularly nice bundles,
this was proved by Bismut and the author in~\cite{BGmorse}
and~\cite{G1}, \cite{G2}.
We will establish the general case in~\cite{G3}.
Igusa also gave a set of axioms in~\cite{Iax} that
characterise~$\tau(E/B;F)$ and hopefully also~$\mathcal T(E/B;F)$
when~$F$ is trivial.
Badzioch, Dorabia\l a and Williams recently gave a cohomological version
of the smooth Dwyer-Weiss-Williams torsion in~\cite{BDW}.
Together with Klein, they proved in~\cite{BDKW}
that it satisfies Igusa's axioms as well.
On the other hand, the other two torsions in~\cite{DWW}
are definitely coarser than Bismut-Lott and Igusa-Klein torsion,
because they do not depend on the differentiable structure.
They might however be related to the Bismut-Lott or Igusa-Klein torsion
of a virtual flat vector bundle~$F$ of rank zero,
see Remark~\ref{Rem7.5} below.

Let us now recall one of the most import applications
of higher torsion invariants.
It is possible to construct two smooth manifold bundles~$p_i\colon E_i\to B$
for~$i=0$, $1$ with diffeomorphic fibres,
such that there exists a homeomorphism~$\varphi\colon E_0\to E_1$
with~$p_0=p_1\circ\varphi$ and with a lift to an isomorphism of vertical
tangent bundles,
but no such diffeomorphism.
The first example of such bundles~$p_i$ was constructed by Hatcher,
and it was later proved by B\"okstedt that~$p_0$ and~$p_1$ are
not diffeomorphic in the sense above~\cite{Boek}.
Igusa showed in~\cite{Ibuch}
that the higher torsion invariants~$\tau(E_i/B;\C)$ differ,
and by~\cite{G1},
the Bismut-Lott torsions~$\mathcal T(E_i/B;\C)$ differ as well.
Hatcher's example can be generalised to construct many different smooth
structures on bundles~$p\colon E\to B$.
We expect that higher torsion invariants distinguish
many of these different structures, but not all of them.

One may wonder why one wants to consider so many different higher torsion
invariants,
in particular, if some of them are conjectured to provide the same information.
We will see that different constructions of these invariants give rise to
different applications.
Since Hatcher's example and its generalisations come with natural
fibrewise Morse functions,
the difference of the Igusa-Klein torsions of different smooth structures
is sometimes easy to compute.
Due to Igusa's axiomatic approach,
one can also understand the topological meaning of Igusa-Klein torsion.
On the other hand, one can classify smooth structures on a topological manifold
bundle~$p\colon E\to B$ in a more abstract way
as classes of sections in a certain bundle of classifying spaces over~$B$.
These section spaces fit well into the framework of generalised Euler
characteristics and Dwyer-Weiss-Williams torsion.
But some extra work is necessary to recover cohomological information
from this approach.

Finally, Bismut-Lott torsion is defined using the language of local index
theory.
The proofs of some interesting properties of Bismut-Lott torsion
were inspired by parallel results in the setting of the classical Atiyah-Singer
family index theorem or the Grothendieck-Riemann-Roch theorem in Arakelov
geometry.
Bismut-Lott torsion is defined for
any flat vector bundle~$F\to E$, whereas Igusa-Klein torsion
and Dwyer-Weiss-Williams torsion can only be defined if the
fibrewise cohomology is of a special type.
This makes Bismut-Lott torsion useful for other applications,
for example in the definition of a secondary $K$-theory by Lott~\cite{Lsec}.
Heitsch and Lazarov generalised Bismut-Lott torsion
to foliations~\cite{HL},
so one may try to use it to detect different smooth structures on a given
foliation, which induce the same structures on the space of leaves.
Finally,
Bismut and Lebeau recently defined higher torsion invariants
using a hypoelliptic Laplacian on the cotangent bundle~\cite{Bhypo},
\cite{BLprep}. Conjecturally, this torsion can give some information
about the fibrewise geodesic flow.

This overview is organised as follows.
We start by discussing the index theorem for flat vector bundles
by Bismut and Lott in Section~\ref{sect1}.
In Sections~\ref{sect2} and~\ref{sect3},
we introduce Bismut-Lott torsion and state some properties and applications
that are inspired by local index theory.
In Section~\ref{sect4} and~\ref{sect5},
we introduce Igusa-Klein torsion and relate it to Bismut-Lott torsion
using two different approaches.
Section~\ref{sect6} is devoted to generalised Euler characteristics
and Dwyer-Weiss-Williams torsion.
In Section~\ref{sect7},
we discuss smooth structures on fibre bundles and a possible generalisation
to foliations.
Finally,
we sketch the hypoelliptic operator
on the cotangent bundle and its torsion due to Bismut and Lebeau in Section~\ref{sect8}.

We have tried to keep the notation and the normalisation of the invariants
consistent throughout this paper; as a result, both will
disagree with most of the references.
In particular, we use the Chern normalisation of~\cite{BGmorse},
which is the only normalisation for which Theorem~\ref{thm3.6} and a few
other results hold.
To keep this paper reasonably short,
only the most basic versions of some of the theorems on higher torsion
will be explained.
Thus we will not discuss some non-trivial generalisations of the theorems below
to fibre bundles with group actions.
We will also only give hints towards the relation with the classical
Atiyah-Singer family index theorem or the Grothendieck-Riemann-Roch theorem in
algebraic geometry.
Finally,
we will not discuss the interesting refinements and generalisations
of classical Franz-Reidemeister torsion and Ray-Singer
torsion for single manifolds that have been invented in the last few years.

\subsubsection*{Acknowledgements}
This paper is a somewhat extended version of a series of lectures
at the Chern Institute at Tianjin in~2007,
whose support and hospitality we highly appreciated.
The author was supported in part by the DFG special programme
``Global Differential Geometry''.

We are grateful to J.-M. Bismut for introducing us to higher torsion,
and also to U. Bunke, W. Dorabia\l a, K. Igusa, K. K\"ohler, X. Ma, B. Williams
and W. Zhang,
from whom we learned many different aspects of this intriguing subject.
We also thank the anonymous referee for her or his helpful comments.

\section{An Index Theorem for Flat Vector Bundles}\label{sect1}

There exists a theory of characteristic classes of flat vector bundles that is parallel to the theory of Chern classes and Chern-Weil differential forms. These classes have been constructed by Kamber and Tondeur~\cite{KT}, and are closely related to the classes used by Borel~\cite{Bo} to study the algebraic $K$-theory of number fields.

Analytic torsion forms made their first appearance in a local index theorem for these Kamber-Tondeur classes by Bismut and Lott~\cite{BL}. Refinements of this theorem have later been given by Dwyer, Weiss and Williams~\cite{DWW} and by Bismut~\cite{Bccs} and Ma and Zhang~\cite{MZ1}.

\subsection{Characteristic classes for flat vector bundles}\label{sect1.1}
Before we introduce Kamber-Tondeur forms, let us first recall classical Chern-Weil theory. Let~$V \to M$ be a complex vector bundle, and let~$\nabla^{V}$ be a connection on~$V$ with curvature~$(\nabla^{V})^{2} \in \Omega^{2}(M;{\rm End} V)$. Then one defines the Chern character form
\begin{equation}\label{eq1.1}
\ch\bigl(V,\nabla^{V}\bigr) = \tr_{V}\biggl(e^{-\tfrac{(\nabla^{V})^{2}}{2\pi i}}\biggr) \in \Omega^{\rm even}(M;\C)\;.
\end{equation}
This form is closed because the covariant derivative~$[\nabla^{V},(\nabla^{V})^{2}]$ of the curvature vanishes by the Bianchi identity, so
\begin{equation}\label{eq1.2}
  d\,\ch\bigl(V,\nabla^{V}\bigr)
  = \tr_{V}\biggl(\biggl[\nabla^{V},
	e^{-\tfrac{(\nabla^{V})^{2}}{2\pi i}}\biggr]\biggr) = 0\;.
\end{equation}

If~$\nabla^{V,0}$ and~$\nabla^{V,1}$ are two connections on~$V$, one can choose a connection~$\nabla^{\tilde{V}}$ on the natural extension~$\tilde{V}$ of~$V$ to~$M \times [0,1]$ with~$\nabla^{\tilde{V}}\mid_{M \times \{i\}} = \nabla^{V,i}$ for~$i = 0$, $1$. Stokes' theorem then implies
\begin{equation}\label{eq1.3}
  \begin{gathered}
    \ch\bigl(V,\nabla^{V,1}\bigr) - \ch\bigl(V,\nabla^{V,0}\bigr)
    = d\,\widetilde{\ch}\bigl(V,\nabla^{V,0},\nabla^{V,1}\bigr), \\
    \mbox{with}\qquad\widetilde{\ch}\bigl(V,\nabla^{V,0}, \nabla^{V,1}\bigr)
    = \int _{0}^{1}\iota_{\tfrac{\partial}{\partial t}}
      \ch\Bigl(\tilde{V},\nabla^{\tilde{V}}\Bigr)\, dt.
  \end{gathered}
\end{equation}
Thus, the class~$\ch(V)$ of~$\ch(V,\nabla^{V})$ in de Rham cohomology is independent of~$\nabla^{V}$. Moreover, $\widetilde{\ch}(V,\nabla^{V,0},\nabla^{V,1})$ is independent of the choice of~$\nabla^{\tilde{V}}$ up to an exact form.

Now let~$F \to M$ be a flat vector bundle, so~$F$ comes with a fixed connection~$\nabla^{F}$ such that~$(\nabla^{F})^{2} = 0$. We choose a metric~$g^{F}$ on~$F$ and define the adjoint connection~$\nabla^{F,*}$ with respect to~$g^{F}$ such that
\begin{equation}\label{eq1.4}
dg(v,w) = g\bigl(\nabla^{F}v,w\bigr) + g\bigl(v,\nabla^{F,*}w\bigr)
\end{equation}
for all sections~$v$, $w$ of~$F$. Then the form
\begin{equation}\label{eq1.5}
  \cho\bigl(F,g^{F}\bigr)
  = \pi i\, \widetilde{\ch}\bigl(F,\nabla^{F},\nabla^{F,*}\bigr)
  \in \Omega^{\mathrm{odd}}(M;\R)
\end{equation}
is real, odd and also closed, because
\begin{equation}\label{eq1.6}
  d\, \cho\bigl(F,g^{F}\bigr)
  =\pi i\,\ch\bigl(F,\nabla^{F,*}\bigr)-\pi i\,\ch\bigl(F,\nabla^{F}\bigr)=0\;.
\end{equation}
Clearly, if~$g^{F}$ is parallel with respect to~$\nabla^{F}$, then~$\cho(F,g^{F}) = 0$.

Let~$g^{F,0}$, $g^{F,1}$ be two metrics on~$F$.
Proceeding as in~\eqref{eq1.3},
one constructs a form~$\chosl(F,g^{F,0}, g^{F,1})\in\Omega^{\mathrm{even}}(M)$
such that
\begin{equation}\label{eq1.7}
    \cho\bigl(F,g^{F,1}\bigr) - \cho\bigl(F,g^{F,0}\bigr)
    = d\,\chosl\bigl(F,g^{F,0},g^{F,1}\bigr)\;.
\end{equation}
So again, the de Rham cohomology class~$\cho(F)$ of~$\cho(F,g^{F})$ does not depend on the choice of metric~$g^{F}$ --- but of course, it depends on the flat connection~$\nabla^{F}$. Note that the form~$\chosl(F,g^{F,0},g^{F,1})$ is again naturally well-defined up to an exact form.

\begin{Definition}\label{Def1.0a}
The forms~$\cho_{k}(F,g^{F})=\cho(F,g^{F}) \in \Omega^{2k-1}(M)$ are called {\em Kamber-Tondeur forms}, and their classes~$\cho_{k}(F) \in H^{2k-1}(M;\R)$ are called {\em Kamber-Tondeur classes} or {\em Borel classes}.
\end{Definition}

Note that in the literature, there are at least three different normalisations of these classes. There are however good reasons to stick to the normalisation here, see section~\ref{sect3.4}.

For later reference, we give a more explicit construction of the Kamber-Tondeur forms. If we define a connection~$\nabla^{\tilde{F}}$ over~$p \colon M \times [0,1] \to M$ that interpolates between~$\nabla^{F}$ and~$\nabla^{F,*}$ by
\begin{equation}\label{eq1.7a}
\nabla^{\tilde{F}} = (1-t)\,p^{*}\nabla^{F} + tp^{*}\nabla^{F,*}\;,
\end{equation}
then by flatness of~$\nabla^{F}$ and~$\nabla^{F,*}$,
\begin{equation}\label{eq1.7b}
\bigl(\nabla^{\tilde{F}}\bigr)^{2} = -t(1-t)\,p^{*}\bigl(\nabla^{F,*}-\nabla^{F}\bigr)^{2} - p^{*}\bigl(\nabla^{F,*} - \nabla^{F}\bigr)\, dt\;.
\end{equation}
From this formula and~\eqref{eq1.3}, \eqref{eq1.5} one deduces that there exist rational multiples~$c_{k}$ of~$(2\pi i)^{k}$ such that
\begin{equation}\label{eq1.7c}
\begin{gathered}
\cho\bigl(F,g^{F}\bigr) = \sum_{k=0}^{\infty} c_{k}\tr_F\bigl(\omega(F,g^{F})^{2k+1}\bigr)\;, \\
\mbox{with}\qquad\omega\bigl(F,g^{F}\bigr) = \nabla^{F,*} - \nabla^{F}
=(g^F)^{-1}\,[\nabla^F,g^F] \in \Omega^{1}(M;\operatorname{End} V)\;.
\end{gathered}
\end{equation}
Bismut and Lott use the real, odd and closed differential forms
\begin{equation}\label{eq1.7d}
\tr_F\biggl(\omega\bigl(F,g^{F}\bigr)e^{\tfrac{\omega(F,g^{F})^{2}}{2\pi i}}\biggr)
\end{equation}
and their cohomology classes instead of~$\cho$, which is more convenient for some of the following constructions. It is not hard to see that these forms are given by a similar formula as~\eqref{eq1.7c}, but with different constants~$c_{k} \in (2\pi i)^{k} \Q$.
We prefer the Chern normalisation given by~\eqref{eq1.5} for reasons
explained in Remark~\ref{rem3.6a}.

The Chern-Weil classes like~$\ch(V)$ vanish whenever~$V$ admits a flat connection. Similarly, the classes~$\cho(F)$ vanish whenever~$F$ admits a $\nabla^{F}$-parallel metric. We will see that there are more analogies between these constructions. A good overview can be found in the introduction to~\cite{Lsec}.

\subsection{The cohomological index theorem}\label{1.2}
The central theme in~\cite{BL} is a family index theorem for flat vector bundles in terms of their Kamber-Tondeur classes.
The analytic index in question is given by fibrewise cohomology.
More precisely, let~$p \colon E \to B$ be a smooth proper submersion,
in other words, a smooth fibre bundle with $n$-dimensional compact fibres,
to be denoted~$M$.
Let~$(F,\nabla^{F})$ be a flat vector bundle
 then we consider the vector bundles~$H^{k}(E/B;F)\to B$,
whose fibres over~$x \in B$ are given as the twisted de Rham cohomology
\begin{equation}\label{eq1.9}
H^{k}(E/B;F)_{x} = H^{k}\bigl(\Omega^{\bullet}(E_{x};F|_{E_x}),\nabla^{F}\bigr)\;.
\end{equation}
The bundles~$H^{k}(E/B;F)$ naturally carry the Gau\ss-Manin connection~$\nabla^{H}$, which is again flat. The analytic index is thus given by the virtual flat vector bundle
\begin{equation}\label{eq1.10}
H(E/B;F) = \bigoplus_{k=0}^{\dim M}(-1)^{k}H^{k}(E/B;F)\;. 
\end{equation}

The topological index is given by the Becker-Gottlieb transfer of~\cite{BeG1}.
Recall that the Becker-Gottlieb transfer is given as a stable homotopy class of
maps~$tr_{E/B}\colon S^\bullet B_+\to S^\bullet E_+$.
It acts on de Rham cohomology by
\begin{equation}\label{eq1.11}
tr_{E/B}^*\,\alpha = \int _{E/B} e(TM)\, \alpha \in H^{k}(B;\R)
\end{equation}
for all~$\alpha \in H^{k}(E;\R)$,
where~$e(TM)\in H^n(E;o(TM)\otimes\R)$ denotes the Chern-Weil theoretic Euler class
of the vertical tangent bundle~$TM = \ker\, dp \subset TE$,
and~$\int _{E/B}$ denotes integration over the fibre.
Here is a cohomological version of the family index theorem.

\begin{Theorem}[Bismut and Lott~\cite{BL}]\label{thm1.1}
For all smooth proper submersions~$p \colon\penalty-50  E\to B$ and all flat vector bundles~$F \to E$,
\begin{equation}\label{eq1.12}
\cho(H(E/B;F)) = tr_{E/B}^*\,\cho(F) \in H^{\mathrm{odd}}(B,\R)\;.
\end{equation}
\end{Theorem}
One notes that~$tr_{E/B}^*$ preserves the degree of differential forms and cohomology classes. For this reason, an analogous result holds for the classes constructed in~\eqref{eq1.7d}, and in fact for all classes of the form~\eqref{eq1.7c}, independent of the choice of the constants~$c_{k}$.

The cohomological index theorem can be refined as follows,
see also Section~\ref{sect3.5}.
Following~\cite{CrS},
to a vector bundle~$V \to M$ with connection~$\nabla^{V}$,
one associates a Cheeger-Simons
differential character~$\widehat{\ch}(V,\nabla^{V})$,
from which both the rational Chern
character~$\ch(V)\in H^{\mathrm{even}}(M;\Q)$
and the Chern-Weil form~$\ch(V,\nabla^{V}) \in \Omega^{\mathrm{even}}(M)$
can be read off.
If~$\nabla^{V}$ is a flat connection,
then~$\widehat{\ch}(V,\nabla^{V})$ becomes a cohomology class
in~$H^{\mathrm{odd}}(M;\C/\Q)$.
It has already been observed in~\cite{BL} that its imaginary part is given by
\begin{equation}\label{eq1.16a}
\Im\,\widehat{\ch}\bigl(V,\nabla^{V}\bigr) = \cho(V) \in H^{\mathrm{odd}}(M;\R)\;.
\end{equation}

\begin{Theorem}[Bismut~\cite{Bccs}, Ma and Zhang~\cite{MZ1}]\label{thm1.1a}
For all smooth proper submersions and all flat vector bundles~$F \to E$,
\begin{equation}\label{eq1.16b}
\widehat{\ch}\bigl(H(E/B;F),\nabla^{H}\bigr) = tr_{E/B}^*\,\widehat{\ch}\bigl(F,\nabla^{F}\bigr) \in H^{\mathrm{odd}}(B;\C/\Q)\;.
\end{equation}
\end{Theorem}

It is natural to ask if the same theorem holds on the level
of flat vector bundles on~$B$.
A flat vector bundle~$F\to E$, or more generally,
a bundle of finitely generated projective $R$-modules for some ring~$R$,
is classified by a map from~$E$
to the classifying space~$BGL(R)\times K_0(R)$.
Following Quillen, there is a natural map from~$BGL(R)$
to the algebraic $K$-theory space~$K(R)$.
Thus,
we may associate to~$F$ the corresponding homotopy class~$[F]$
of maps from~$E$ to~$K(R)$,
which is slightly coarser than the class of~$F$
in the $K$-theory of finitely generated projective $R$-module bundles on~$E$.

\begin{Theorem}[Dwyer, Weiss and Williams~\cite{DWW}]\label{thm1.2}
If~$p\colon E\to B$ is a bundle of smooth closed manifolds, then
\begin{equation}\label{eq1.13}
[H(E/B;F)] = tr_{E/B}^*\,[F]
\end{equation}
in the homotopy classes of maps~$B\to K(R)$.
\end{Theorem}

Although both sides of~\eqref{eq1.13} exist in a much more general situation,
the smooth bundle structure is needed in the proof of the theorem,
see section~\ref{sect6.1} below, in particular Theorem~\ref{thm6.3}.
Theorem~\ref{thm1.1} 
can be deduced from Theorem~\ref{thm1.2} because the class~$\cho$
can already be defined on~$K(R)$.

\subsection{A refined index theorem}\label{sect1.3} 
There is another possible refinement of Theorem~\ref{thm1.1}, where one replaces de Rham cohomology classes by differential forms. For this, one first chooses metrics~$g^{TM}$ and~$g^{F}$ on the bundles~$TM \to E$ and~$F \to E$, and a horizontal complement~$T^{H}E$ of~$TM \subset TE$. These data give rise to a natural connection~$\nabla^{TM}$ on~$TM$ by~\cite{B}. Thus, one can consider the Chern-Weil theoretic Euler form~$e(TM,\nabla^{TM})$.

We also have a natural decomposition
\begin{equation}\label{eq1.14}
\Omega^{\bullet}(E;F) = \Omega^{\bullet}(B;\Omega^{\bullet}(E/B;F))
\end{equation}
using~$TE = T^{H} E \oplus TM$, and an $L^{2}$-metric on the infinite dimensional bundle~$\Omega^{\bullet}(E/B;F) \to B$ of vertical forms twisted by~$F$. Regarding~$H^{\bullet}(E/B;F)$ as the subbundle of fibrewise harmonic forms, we get a metric~$g_{L^{2}}^{H}$ on~$H^{\bullet}(E/B;F)$. Bismut and Lott now construct a form~${\mathcal T}(T^{H}E,g^{TM},g^{F})$ on~$B$ that depends natural on the data, the {\em analytic torsion form,\/} see Section~\ref{sect2.2} below.

\begin{Theorem}[Bismut and Lott~\cite{BL}]\label{thm1.3}
  In the situation above,
  \begin{equation}\label{eq1.20}
    d\,\mathcal T\bigl(T^{H}E,g^{TM},g^{F}\bigr)
    = \int _{E/B} e\bigl(TM,\nabla^{TM}\bigr)\,\cho\bigl(F,g^{F}\bigr)
      - \cho\bigl(H,g_{L^{2}}^{H}\bigr)\;.
  \end{equation}
\end{Theorem}

In the theory of flat vector bundles, this result plays the same role as the $\eta$-forms in the heat kernel proof of the classical family index
theorem~\cite{B}, \cite{BGV}, see also~\cite{BCeta} and~\cite{Dai}.
The holomorphic torsion forms similarly arise in a double transgression
formula~\cite{BiK} in the Riemann-Roch-Grothendieck theorem
for proper holomorphic submersions in K\"ahler geometry.
This analogy with $\eta$-forms and holomorphic torsion forms has
inspired most of the constructions and results of the following two sections.

\section{Construction of the Bismut-Lott torsion}\label{sect2}
In this section, we recall the construction of the torsion forms occurring in Theorem~\ref{thm1.3}. As in~\cite{BL}, we start with a finite-dimensional toy model that will be of independent interest. We then present the original construction of~${\mathcal T}(T^{H}M,g^{TM},g^{F})$ by Bismut and Lott, and also a construction using $\eta$-forms by Ma and Zhang.

\subsection{A finite-dimensional model}\label{sect2.1}
Consider flat vector bundles~$V^{k} \to M$ and parallel vector bundle homomorphisms~$a_{k} \colon V^{k} \to V^{k+1}$, such that
\begin{equation}\label{eq2.1}
  \begin{CD}
    0@>>>V^0@>a_0>>V^1@>a_1>>\cdots@>a_{n-1}>>V^n@>>>0
  \end{CD}
\end{equation}
forms a cochain complex over each point in~$M$. Then
\begin{equation}\label{eq2.2}
A' = \nabla^{V} + a 
\end{equation}
is a superconnection, which is flat because
\begin{equation}\label{eq2.3}
(A')^{2} = a^{2} + \bigl[\nabla^{V},a\bigr] + \bigl(\nabla^{V}\bigr)^{2}\;,
\end{equation}
and each term on the right hand side vanishes by assumption. We will call the pair~$(V,\nabla^{V}+a)$ a {\em parallel family of\/} (finite-dimensional) {\em cochain complexes}. 

If we fix a metric~$g^{V^k}$ on each~$V^{k}$, we can consider the adjoint connection~$\nabla^{V,*}$ as in~\eqref{eq1.4}, and let~$a_{k}^{*}\colon V^{k+1} \to V^{k}$ be the adjoint of~$a_{k}$ with respect to~$g^{V^k}$ and~$g^{V^{k+1}}$. Then we obtain another flat superconnection
\begin{equation}\label{eq2.4}
A'' = \nabla^{V,*} + a^{*}.
\end{equation}
As in Hodge theory, the fibrewise cohomology of~$(V,a)$ is represented by~$H = \ker(a+a^{*}) \subset V$. Projection of~$\nabla^{V}$ onto~$H$ defines a connection~$\nabla^{H}$ on~$H$. One checks that~$\nabla^{H}$ is independent of~$g^{V}$, and in fact, $\nabla^{H}$ is the natural Gau\ss-Manin connection.
Let~$g_{V}^{H}$ denote the restriction of~$g^{V}$ to~$H$.

Bismut and Lott then define a differential form~$T(\nabla^{V}+a,g^{V}) \in \Omega^{\mathrm{even}}(M)$ and obtain a finite-dimensional analogue of Theorem~\ref{thm1.3}. 

\begin{Theorem}[Bismut and Lott, \cite{BL}]\label{thm2.1} In the situation above, 
\begin{equation}\label{eq2.5}
dT\bigl(\nabla^{V}+a,g^{V}\bigr) = \cho\bigl(V,g^{V}\bigr) - \cho\bigl(H,g_{V}^{H}\bigr)\;.
\end{equation}
\end{Theorem}

The core of the proof is the construction of~$T(\nabla^{V}+a,g^{V})$ that we now describe. On the pullback~$\tilde{V}$ of~$V$ to~$\tilde M=M \times (0,\infty)$, we introduce two flat superconnections 
\begin{equation}\label{eq2.6}
\begin{aligned}
\tilde{A}' &= \nabla^{V} + \sqrt{t} a - \frac{N^{V}}{2t}\, dt\;, \\
\tilde{A}'' &= \nabla^{V,*} + \sqrt{t} a^{*} + \frac{N^{V}}{2t}\, dt\;,
\end{aligned}
\end{equation}
where~$N^{V} \in \operatorname{End} V$ acts on~$V^{k}$ as multiplication by~$k$. The difference of the two superconnections above is an endomorphism
\begin{equation}\label{eq2.7}
\tilde{X} = \tilde{A}'' - \tilde{A}' = \omega\bigl(V,g^{V}\bigr) + \sqrt{t}(a^{*}-a) + \frac{N^{V}}{t}\, dt \in \Omega^{\bullet}\bigl(\tilde M,\operatorname{End}\tilde{V}\bigr)\;.
\end{equation}
We also define the supertrace by
\begin{equation}\label{eq2.8}
\str_{V} = \tr_{V} \circ (-1)^{N^V}\colon \Omega^{\bullet}(\punkt,\operatorname{End} V) \to \Omega^{\bullet}(\punkt)\;.
\end{equation}
For convenience, we stick to the conventions of~\cite{BL}. In analogy with~\eqref{eq1.7d}, the form
\begin{equation}\label{eq2.9}
(2 \pi i)^{\frac{1-N^M}{2}}\str{V}\Bigl(\tilde{X}e^{\tilde{X}^{2}}\Bigr) \in \Omega^{\mathrm{odd}}(\tilde M)
\end{equation}
is real, odd and closed. By~\eqref{eq2.7}, we have
\begin{equation}\label{eq2.10}
\lim_{t \to 0} \str{V}\left(\tilde{X}e^{\tilde{X}^{2}}\right)\Big|_{M \times \{t\}} = \str_{V}\left(\omega\bigl(V,g^{V}\bigr)e^{\omega(V,g^{V})^{2}}\right)\;.
\end{equation}
To understand the limit for~$t \to \infty$, note that~$a^{*} - a$ is a skew-adjoint operator. In particular, the ``finite dimensional Laplacian''~$-(a^{*}-a)^{2}$ has nonnegative eigenvalues, and its kernel is given by the ``harmonic elements''~$H$. In particular, the ``heat operator''~$e^{t(a^{*}-a)^{2}}$ converges to the orthogonal projection onto~$H$ as~$t$ tends to infinity. More generally, it is proved in~\cite{BL} that
\begin{equation}\label{eq2.11}
\lim_{t \to \infty}\str_{V}\left(\tilde{X}e^{\tilde{X}^{2}}\right)\Big|_{M \times\{t\}} = \str_{H} \left(\omega\bigl(H,g_{V}^{H}\bigr)e^{\omega(H,g_{V}^{H})^{2}}\right).
\end{equation}
Because the form in~\eqref{eq2.9} is closed, the forms in~\eqref{eq2.10} and~\eqref{eq2.11} belong to the same cohomology class. Thus we have already proved a finite-dimensional version of Theorem\ref{thm1.1}. To define the torsion form, we have to integrate the form in~\eqref{eq2.9} over~$(0,\infty)$. We note that
\begin{equation}\label{eq2.12}
  \iota_{\tfrac{\partial}{\partial t}}
    \str_{V}\Bigl(\tilde{X}e^{\tilde{X}^{2}}\Bigr)\Bigr|_{M \times\{t\}}
  =\str_{V}\left(\frac{N^{V}}{t}(1+2\tilde{X}^{2})
    \,e^{\tilde{X}^{2}}\right)\biggr|_{M \times\{t\}}.
\end{equation} 
Unfortunately, the integral over~\eqref{eq2.12} diverges both for~$t \to 0$ and for~$t \to \infty$. However, the divergence can be compensated easily. For any~$\Z$-graded vector bundle~$V$, we define
\begin{equation}\label{eq2.13}
\chi(V) = \sum_{k}(-1)^{k}\rk V^{k} \quad\quad\text{and}\quad \quad \chi'(V) = \sum_{k}(-1)^{k} k \rk V^{k}.
\end{equation}
Then it is proved in~\cite{BL} that the integral
\begin{multline}\label{eq2.14}
\int _{0}^{\infty} \biggl((2\pi i)^{-\frac{N^M}{2}}\str_{V}\left(N^{V}(1+2\tilde{X}^{2})e^{\tilde{X}^{2}}\right) - \chi'(H) \\
- (\chi'(V)-\chi'(H))(1-2t)e^{-t}\biggr)\,\frac{dt}{t} \in \Omega^{\mathrm{even}}(M)
\end{multline}
converges and gives a torsion form for the characteristic classes considered in~\eqref{eq1.7d}. Adjusting the coefficients~$c_{k}$ in~\eqref{eq1.7c}, we obtain the form~$T(\nabla^{V}+a,g^{V})$ needed for Theorem~\ref{thm2.1}.

\begin{Definition}\label{Def2.2} The {\em Bismut-Lott torsion} of the parallel family of cochain complexes~$(V,\nabla^{V}+a)$ is defined as
\begin{multline}\label{eq2.15}
  T\bigl(\nabla^{V}+a,g^{V}\bigr)
  =-\int _{0}^{1}\smash{\biggl(\frac{s(1-s)}{2\pi i}\biggr)^{\tfrac{N^M}{2}}}
    \int _{0}^{\infty}\biggl(\str_{V}\left(N^{V}
      (1+2\tilde{X}^{2})e^{\tilde{X}^{2}}\right) \\
  - \chi'^{H} -(\chi'(V)-\chi'(H))(1-2t^{2})e^{-t^{2}}\biggr)\,\frac{dt}{2t}\, ds
    \quad\in \Omega^{\mathrm{even}}(M).
\end{multline}
\end{Definition}

\begin{proof}[Proof of Theorem~\ref{thm2.1}]
  Let~$N^{M}$ act on~$\Omega^{k}(M)$ as multiplication by~$k$.
  Because the form~\eqref{eq2.9} is closed, it follows from~\eqref{eq2.10} and~\eqref{eq2.11} that
  \begin{equation}\label{eq2.16}
    \begin{aligned}
      dT\bigl(\nabla^{V}+a,g^{V}\bigr) &= \frac{1}{2} \int _{0}^{1}
        \smash{\left(\frac{s(1-s)}{2\pi i}\right)^{\tfrac{N^{M}-1}{2}}}
	\biggl(\lim_{t\to 0} \str_{V}\left(\tilde{X}e^{\tilde{X}^{2}}\right)
	\Bigr|_{M \times \{t\}} \\
      &\kern6em- \lim_{t \to \infty} \str_{V}\left(\tilde{X}e^{\tilde{X}^{2}}\right)\Bigr|_{M \times \{t\}}\biggr) \\
      &= \cho\bigl(V,g^{V}\bigr) - \cho\bigl(H,g_{V}^{H}\bigr). \qedhere
    \end{aligned}
  \end{equation}
\end{proof}

\begin{Remark}\label{Rem2.3}
The correction terms in Definition~\ref{Def2.2} are constant and only affect the Bismut-Lott torsion in degree 0. They are chosen such that
\begin{equation}\label{eq2.17}
\begin{aligned}
T\bigl(\nabla^{V}+a,g^{V}\bigr)^{[0]}_{x} &= \frac{1}{2}\sum_{k}(-1)^{k} k \log \det\left(-(a^{*}-a)^{2}\big|_{V^{k}\cap H^{k\perp}}\right) \\
&= \frac{1}{2} \sum_{k}(-1)^{k}\log \det\left(aa^{*}\big|_{V^{k}\cap \im a}\right).
\end{aligned}
\end{equation}
But this is just one way to represent the Franz-Reidemeister torsion of the cochain complex~$(V_{x},a)$ with metric~$g^{V}$ for~$x \in M$. Hence~$T(\nabla^{V}+a,g^{V})$ is called a ``higher torsion form''.
\end{Remark}

\subsection{The Bismut-Lott torsion form}\label{sect2.2}
As in~\cite{BL}, section 3, we now translate the construction of~$T(\nabla^{V}+a,g^{V})$ to the infinite-dimensional family of fibrewise de Rham complexes.

Let~$p \colon E \to B$ be a smooth proper submersion with typical fibre~$M$, and let~$TM = \ker dp \subset TE$. As in Section~\ref{sect1.3}, we fix~$T^{H}E \subset TE$ such that~$TE = TM \oplus T^{H}E$. Because~$T^{H}E \cong p^{*}TB$, we can identify vector fields on~$B$ with their pullback to~$E$, which we call basic vector fields.

Let~$F \to E$ be a flat vector bundle, then we may regard the flat connection~$\nabla^{F}$ as a differential on the total complex~$\Omega^{\bullet}(E;F)$. Using the splitting~\eqref{eq1.14}, we may also regard~$\nabla^{F}$ as a superconnection on the infinite-dimensional bundle~$\Omega^{\bullet}(E/B;F) \to B$ with
\begin{equation}\label{eq2.18}
{\mathbb A}' = \nabla^{F} = d^{M} + \nabla^{\Omega^{\bullet}(E/B;F)} + \iota_{\Omega}
\end{equation}
by~\cite{BGV}. Here, $d^{M}$ denotes the fibrewise differential on~$\Omega^{\bullet}(E/B;F)$, $\nabla^{\Omega^{\bullet}(E/B;F)}$ is the connection induced by the Lie derivative by basic vector fields, and~$\Omega$ is the vertical component of the Lie bracket of two basic vector fields on~$E$. 

In analogy with~\eqref{eq2.4}, we also define an adjoint superconnection
\begin{equation}\label{eq2.19}
{\mathbb A}'' = d^{M,*} + \nabla^{\Omega^{\bullet}(E/B;F),*}+ \eps_{\Omega} 
\end{equation}
with respect to the fibrewise~$L^{2}$-metric~$g_{L^{2}}$
on~$\Omega^{\bullet}(E/B;F)$.
Let~$\tilde{B} = B \times (0,\infty)$, $\tilde{E} = E \times (0,\infty)$
and~$\tilde{F} = F \times (0,\infty)$,
and let~$t$ be the coordinate of~$(0,\infty)$.
Then we define superconnections
\begin{equation}\label{eq2.20}
  \begin{aligned}
    \tilde{{\mathbb A}}'
    &= \sqrt{t}\, d^{M}
      + \nabla^{\Omega^{\bullet}(\tilde{E}/\tilde{B};\tilde{F})}
      + \frac{1}{\sqrt{t}}\,\iota_{\Omega}
      - \frac{N^{\tilde{E}/\tilde{B}}}{2t}\, dt\;, \\
    \tilde{{\mathbb A}}''
    &= \sqrt{t}\, d^{M,*}
      + \nabla^{\Omega^{\bullet}(\tilde{E}/\tilde{B};\tilde{F}),*}
      + \frac{1}{\sqrt{t}}\,\eps_{\Omega}
      + \frac{N^{\tilde{E}/\tilde{B}}}{2t}\, dt\;,
  \end{aligned}
\end{equation}
where now~$N^{\tilde{E}/\tilde{B}}$ acts on~$\Omega^{k}(\tilde{E}/\tilde{B};\tilde F)$ as multiplication by~$k$. Then
\begin{multline}\label{eq2.21}
  \tilde{\X} = \tilde{\A}'' - \tilde{\A}'
  = \sqrt{t}\bigl(d^{M,*}-d^{M}\bigr)
    +\omega\bigl(\Omega^{\bullet}(\tilde{E}/\tilde{B};\tilde F),g_{L^2}\bigr)\\
    + \frac{1}{\sqrt{t}}(\eps_{\Omega}-\iota_{\Omega})
    + \frac{N^{\tilde{E}/\tilde{B}}}{t}\, dt
  \quad\in \Omega^{\bullet}\bigl(\tilde{B};
    \operatorname{End}\Omega^{\bullet}(\tilde{E}/\tilde{B};\tilde{F})\bigr).
\end{multline}
Note that~$d^{M,*} - d^{M}$ is a skew-adjoint fibrewise elliptic differential operator, whereas the other terms on the right hand side involve no differentiation at all. The operator~$-\tilde{\X}^{2}$ can be regarded as a generalised Laplacian along the fibres of~$p$. If the metric~$g^{F}$ is parallel along the fibres, then~$-\tilde{\X}^{2}$ is precisely the curvature of the Bismut superconnection, which already appeared in the heat equation proof of the Atiyah-Singer families index theorem~\cite{B}. In particular, the fibrewise odd heat operator~$\tilde{\X}e^{\tilde{X}^{2}}$ is well-defined and of trace class. Using Getzler rescaling, one proves
\begin{multline}\label{eq2.22}
  \lim_{t \to 0} \str_{\Omega^{\bullet}(\tilde{E}/\tilde{B};\tilde{F})}
    \Bigl(\tilde{\X}e^{\tilde{\X}^{2}}\Bigr)\Bigr|_{B \times \{t\}}\\
  =\int _{E/B} e\bigl(TM,\nabla^{TM}\bigr)
    \,\str\left(\omega\bigl(F,g^{F}\bigr)e^{\omega(F,g^{F})^{2}}\right)
\end{multline}
in analogy with~\eqref{eq2.10}. Similarly, if we identify~$H = H^{\bullet}(E/B;F)$ with the fibrewise harmonic differential forms, equipped with the restriction~$g_{L^{2}}^{H}$ of the $L^{2}$-metric on~$\Omega^{\bullet}(E/B;F)$, then
\begin{equation}\label{eq2.23}
\lim_{t \to \infty} \str_{\Omega^{\bullet}(\tilde{E}/\tilde{B};\tilde{F})}\left(\tilde{\X}e^{\tilde{\X}^{2}}\right)\big|_{B \times \{t\}} = \str_{H}\left(\omega\bigl(H,g_{L^{2}}^{H}\bigr)e^{\omega(H,g_{L^{2}}^{H})^{2}}\right)
\end{equation}
as in~\eqref{eq2.11}. To obtain the torsion form, we have to take care of some divergent terms and of the coefficients in~\eqref{eq1.7c} as before.

\begin{Definition}\label{Def2.4}
The Bismut-Lott torsion is defined as
\begin{multline}\label{eq2.24}
{\mathcal T}\bigl(T^{H}E,g^{TM},g^{F}\bigr)\\
= -\int _{0}^{1}\left(\frac{s(1-s)}{2\pi i}\right)^{\tfrac{N^{B}}{2}} \int _{0}^{\infty}\biggl(\str_{\Omega^{\bullet}(\tilde{E}/\tilde{B};\tilde{F})}\left(N^{\tilde{E}/\tilde{B}}\bigl(1+2\tilde{\X}^{2}\bigr)e^{\tilde{\X}^{2}}\right)- \chi'(H)  \\
- \Bigl(\frac{\chi(M)\,\dim M \rk F}{2} - \chi'(H)\Bigr)(1-2t)e^{-t}\biggr)\,\frac{dt}{2t}\, ds \quad\in \Omega^{\mathrm{even}}(B)\;.
\end{multline}
\end{Definition}

\begin{proof}[Proof of Theorem~\ref{thm1.3}] As in the proof of Theorem~\ref{thm2.1}, this follows from~\eqref{eq2.22} and~\eqref{eq2.23}, because the form
\begin{equation}\label{eq2.25}
\str_{\Omega^{\bullet}(\tilde{E}/\tilde{B};\tilde{F})}\left(\tilde{\X} e^{\tilde{\X}^{2}}\right) \in \Omega^{\bullet}(B \times (0,\infty);\C)
\end{equation}
is closed. 
\end{proof}

\begin{Remark}\label{Rem2.5} Again, the correction terms in~\eqref{eq2.24} are constant scalars. They are chosen such that
\begin{multline}\label{eq2.26}
{\mathcal T}\bigl(T^{H}E,g^{TM},g^{F}\bigr)_{x}^{[0]}\\
 = \frac{1}{2} \sum_{k=0}^{\dim M}(-1)^{k} k \log\operatorname{Det}\left(-\bigl(d^{M,*}-d^{M}\bigr)^{2}\big|_{\Omega^{k}(M_{x};F)\cap H^{k\perp}}\right)\;,
\end{multline}
where ``Det'' denotes a zeta-regularised determinant.
The right hand side is precisely the Ray-Singer analytic torsion
of the fibre~$M_{x}$.
Hence~${\mathcal T}(T^{H}E, g^{TM},g^{F})$ is called a Bismut-Lott torsion form.
\end{Remark}

\subsection{Elementary Properties}\label{sect2.3}
From Theorem~\ref{thm1.3}, one can derive a variation formula for Bismut-Lott torsion. If we choose~$T_{j}^{H}E$, $g_{j}^{TM}$, $g_{j}^{F}$ for~$j = 0$, $1$, let~$\nabla^{TM,j}$ denote the corresponding connections on~$TM$, and let~$g_{L^{2}}^{H,j}$ denote the corresponding $L^{2}$-metrics on~$H$. As in~\eqref{eq1.3}, there exists a Chern-Simons Euler class~$\tilde{e}(TM,\nabla^{TM,0},\nabla^{TM,1})$ such that
\begin{equation}\label{eq2.27}
  d\tilde{e}\bigl(TM,\nabla^{TM,0},\nabla^{TM,1}\bigr)
  = e\bigl(TM,\nabla^{TM,1}\bigr) - e\bigl(TM,\nabla^{TM,0}\bigr)\;.
\end{equation}

\begin{Theorem}[Bismut and Lott~\cite{BL}]\label{thm2.6} Modulo exact forms on~$B$,
\begin{multline}\label{eq2.28}
{\mathcal T}\bigl(T_{1}^{H}E,g_{1}^{TM},g_{1}^{F}\bigr) - {\mathcal T}\bigl(T_{0}^{H}E,g_{0}^{TM},g_{0}^{F}\bigr)\\
 = \int _{E/B}\Bigl(\tilde{e}\bigl(TM,\nabla^{TM,0},\nabla^{TM,1}\bigl)\,\cho\bigl(F,g_{0}^{F}\bigl)
+ e\bigl(TM,\nabla^{TM,1}\bigl)\,\chosl\bigl(F,g_{0}^{F},g_{1}^{F}\bigl)\Bigr)\\
- \chosl\left(H^{\bullet}(E/B;F),g_{L^{2}}^{H,0}, g_{L^{2}}^{H,1}\right).
\end{multline}
\end{Theorem}
A variation formula like this has already been proved for the Ray-Singer torsion in~\cite{BZ1}.
Theorem~\ref{thm2.6} is a direct consequence of Theorem~\ref{thm1.3}.
Similar variation formulas exist for $\eta$-forms~\cite{BC}
and holomorphic torsion forms~\cite{BiK}.

\begin{Corollary}[Bismut and Lott~\cite{BL}]\label{Cor2.7} If the fibres of~$p \colon E \to B$ are odd-dimensional and~$F \to E$ is fibrewise acyclic, then~${\mathcal T}(T^{H}E,g^{TM},g^{F})$ defines an even cohomology class on~$B$ that is independent of the choices of~$T^{H}E$, $g^{TM}$ and~$g^{F}$.
\end{Corollary}

There is another situation where~${\mathcal T}(T^{H}E,g^{TM},g^{F})$ defines a cohomology class, at least its higher degree components. Assume that~$g_{0}^{F}$ and~$g_{1}^{F}$ are both parallel with respect to~$\nabla^{F}$. Then
\begin{equation}\label{eq2.29}
g_{t}^{F} = (1-t)g_{0}^{F} + t g_{1}^{F}
\end{equation}
is a parallel metric on~$F$ for all~$t \in [0,1]$. Put the metric~$g^{\tilde{F}}|_{F \times \{t\}} = g_{t}^{F}$ on the pullback~$\tilde{F}$ to~$\tilde{E} = E \times [0,1]$, then
\begin{equation}\label{eq2.30}
\omega\bigl(\tilde{F},g^{\tilde{F}}\bigr) = (g_{t}^{F})^{-1} \frac{\partial}{\partial t} g_{t}^{F}\, dt \in \Omega^{1}(E \times [0,1];\operatorname{End} \tilde{F})
\end{equation}
because~$\omega(\tilde{F},g^{\tilde{F}})|_{E \times \{t\}} = 0$ by~\eqref{eq1.7c}. In particular
\begin{equation}\label{eq2.31}
\chosl\bigl(F,g_{0}^{F},g_{1}^{F}\bigr) = c_{0}\int _{0}^{1}\tr_{F}\left((g_{t}^{F})^{-1} \frac{\partial}{\partial t} g_{t}^{F}\right)\, dt \in \Omega^{0}(E)
\end{equation}
is in fact just a constant function on~$E$.

\begin{Definition}\label{Def2.8} If the bundles~$F \to E$ and~$H(E/B;F) \to B$
admit parallel metrics~$g^{F}$ and~$g^{H}$,
one defines the higher analytic torsion or {\em Bismut-Lott torsion\/} as
  \begin{equation}\label{eq2.31a}
    {\mathcal T}(E/B;F)
    = {\mathcal T}\bigl(T^{H}E,g^{TM},g^{F}\bigr)^{[\ge 2]}
      + \chosl\bigl(H,g^{H},g_{L^{2}}^{H}\bigr)^{[\ge 2]}
    \quad\in \Omega^{\ge 2}(B)\;.
  \end{equation}
\end{Definition}

It follows from Theorems~\ref{thm1.3} and~\ref{thm2.6} that~${\mathcal T}(E/B;F)$ defines a cohomology class in~$H^{\ge 2}(B;\R)$ that is independent of~$T^{H}E$, $g^{TM}$, $g^{F}$ and~$g^{H}$, as long as~$g^{F}$ and~$g^{H}$ are parallel metrics.

\section{Properties of Bismut-Lott torsion}\label{sect3}
Since $\eta$-forms, analytic torsion forms and holomorphic torsion forms
are parallel objects in three somewhat similar theories,
one can try to translate any result concerning one of those three objects
into theorems on the other two.
In this section,
we present a few results on higher torsion that where at least partially
motivated by results on $\eta$-forms or on holomorphic torsion forms.
In particular,
we recall results by Ma and Bunke on torsion forms of iterated fibrations,
and of Bunke, Bismut and the author
about the relation with equivariant Ray-Singer torsion.
Most of these theorems have not yet been proved for Igusa-Klein
or Dwyer-Weiss-Williams torsion.
We also discuss Ma and Zhang's construction using $\eta$-invariants
of subsignature operators.

One should mention at this point that in the theory of flat vector bundles,
we are only considering proper submersions.
The reason is that the direct image of a flat  vector bundle
under other maps like open or closed embeddings is in general
not given by a flat vector bundle.
Another reason is that there is no suitable analogue
of the Becker-Gottlieb transfer for general maps.
For this reason, many beautiful results for $\eta$-invariants
and holomorphic torsion have no counterpart for Bismut-Lott torsion. 

\subsection{A transfer formula}\label{sect3.1}
Consider a smooth proper submersion~$p_{1} \colon E \to B$ with typical fibre~$M$ as before, and assume that~$p_{2} \colon D \to E$ is another smooth proper submersion with fibre~$N$. Then~$p_{3} = p_{1} \circ p_{2}$ is again a smooth proper submersion, and its fibre~$L$ maps to~$M$ with fibre~$N$. Let~$F \to D$ be a flat vector bundle, then we have higher direct images
\begin{equation}\label{eq3.1}
\begin{aligned} 
K &= \bigoplus_{k=0}^{\dim N}(-1)^{k} H^{\bullet}(D/E;F) \to E \\
\mbox{and}\quad H &= \bigoplus_{k=0}^{\dim L} (-1)^{k}H^{\bullet}(D/B;F) \to B.
\end{aligned}
\end{equation}
Note that~$H$ is not the higher direct image of~$K$ under~$p_{1}$. Instead, there is a fibrewise Leray-Serre spectral sequence over~$B$ with $E_{2}$-term~$H^{\bullet}(E/B;K)$ that converges to~$H$. Beginning with~$E_{2}$, the higher terms in this spectral sequence are given by parallel families of finite-dimensional cochain complexes~$\left(E_{k},\nabla^{E_k} + d_{k}\right)$ over~$B$. Of course, $E_{n} = E_{\infty}$ and~$d_{n} = 0$ for all sufficiently large~$n$. 

We now choose compatible complements of the vertical tangent bundles for all
three fibrations, fibrewise Riemannian metrics, and a metric on the bundle~$F$.
Again,
these data induce connections on the three vertical tangent bundles~$TM$,
$TN$ and~$TL \cong TN \oplus p_{2}^{*}TM$.
They also induce $L^{2}$-metrics on the flat vector bundles~$H$ and~$E_{k}$
over~$B$ for~$k \ge 2$ and on~$K \to E$.
We need the Chern-Simons Euler form~$\tilde{e}$,
which is constructed in analogy with~$\widetilde{\ch}$ in~\eqref{eq1.3}, and we also need another finite-dimensional torsion form~$T(H,E_{\infty},g^{H},g^{E_\infty})$ relating the filtered flat vector bundle~$H$ to its graded version~$E_{\infty} = E_{n}$ for~$n$ sufficiently large.

\begin{Theorem}[Transfer formula, Ma~\cite{Mfn}]\label{thm3.1} Modulo exact forms on~$B$, we have
\begin{multline}\label{eq3.2}
  {\mathcal T}\bigl(T^{H}D,g^{TL},g^{F}\bigr)
  =\int_{E/B}e\bigl(TM,\nabla^{TM}\bigr)\,{\mathcal T}\bigl(H^{H}D\oplus T^{H}L,g^{TN},g^{F}\bigr)\\
    +{\mathcal T}\bigl(T^{H}E,g^{TM},g^{K}\bigr)
    +\sum_{k=2}^{\infty}T\left(\nabla^{E_k}+d_{k},g^{E_k}\right)
    + T\bigl(H,E_{\infty},g^{H},g^{E_\infty}\bigr) \\
+ \int _{D/B} \tilde{e}\left(TL,\nabla^{TL},\nabla^{TN} \oplus p_{2}^{*}\nabla^{TM}\right)\,\cho\bigl(F,g^{F}\bigr)\;. 
\end{multline}
\end{Theorem}

The first two terms on the right hand side should be regarded as torsion forms of the terms~$E_{0}$ and~$E_{1}$ of the Leray-Serre spectral sequence.
The sum of the torsions of the remaining terms is of course finite.
The theorem says in other words that the analytic torsion form of the total fibration is the sum of the torsion forms of all terms in the Leray-Serre spectral sequence and two natural correction terms.
A similar formula for holomorphic torsion forms has been proved
by Ma~\cite{Mh1}, \cite{Mh2}.
For $\eta$-invariants of signature operators,
an analogous result is due to Bunke and Ma~\cite{BM}.

\subsection{Lott's Secondary \texorpdfstring{$K$}{K}-theory of flat bundles}\label{sect3.2}
In Arakelov geometry, one studies arithmetic Chow groups,
which constitute a simultaneous refinement of classical Chow groups
and of de Rham forms,
see~\cite{SABK} for an introduction.
The central objects in this theory are algebraic vector bundles
over arithmetic schemes,
together with Hermitian metrics on the corresponding holomorphic
vector bundles over the complex points of those schemes,
which form classical complex algebraic varieties.
To construct the ``complex algebraic'' part of the direct image
of such vector bundles,
one needs the holomorphic torsion forms of Bismut and K\"ohler~\cite{BiK}.
To establish elementary properties
of this direct image construction,
one needs deep results on holomorphic torsion forms.
Thus, Arakelov geometry has been one of the main motivations
for the many results on holomorphic torsion by Bismut and others.
For this reason, it is tempting to have a similar theory for
flat vector bundles over smooth manifolds,
where Bismut-Lott torsion plays the role of holomorphic torsion forms.

Lott's $K$-theory of flat vector bundles with vanishing Kamber-Tondeur classes
is a first step in this direction.
But note that there are no objects corresponding to Chow cycles,
and that we can take direct images only for submersions,
for reasons explained at the beginning of this section.
Thus we cannot expect a theory that is as rich as arithmetic Chow theory.
Nevertheless,
some nice results are motivated by Lott's construction.

We consider triples~$(F,g^{F},\alpha)$, where~$F \to M$ is a flat vector bundle, equipped with a metric~$g^{F}$, and~$\alpha \in \Omega^{\mathrm{even}}(M)/d\Omega^{\mathrm{odd}}(M)$ satisfies
\begin{equation}\label{eq3.3}
\cho\bigl(F,g^{F}\bigr) - d\alpha = 0 \in \Omega^{\mathrm{odd}}(M).
\end{equation}
A short exact sequence
\begin{equation}\label{eq3.4}
  \begin{CD}
    0@>>>F_1@>a_1>>F_2@>a_2>>F_3@>>>0
  \end{CD}
\end{equation}
of flat vector bundles and parallel linear maps can be interpreted as a parallel family of acyclic chain complexes~$(F,\nabla^{F}+a)$. Let~$g^{F_1}$, $g^{F_2}$, $g^{F_3}$ be metrics on these bundles. By Theorem~\ref{thm2.1}, the higher torsion form of this family satisfies
\begin{equation}\label{eq3.5}
dT\bigl(\nabla^{F}+a,g^{F}\bigr) = \cho\bigl(F_{2},g^{F_2}\bigr) - \cho\bigl(F_{1},g^{F_1}\bigr) - \cho\bigl(F_{3},g^{F_3}\bigr).
\end{equation}

\begin{Definition}\label{Def3.2} Lott's secondary $K$-group~$\overline{K}{}^0(M)$ is the abelian group generated by triples~$(F,g^{F},\alpha)$ subject to

\begin{enumerate}
\item{} the condition~\eqref{eq3.3}, and
\item{}the relation
\[ T\bigl(\nabla^{F}+a,g^{F}\bigr) = \alpha_{2}-\alpha_{1}-\alpha_{3} \in \Omega^{\mathrm{even}}(M)/d\Omega^{\mathrm{odd}}(M) \]
for each short exact sequence~\eqref{eq3.4}.
\end{enumerate}
\end{Definition}

In fact, Lott considers groups~$\overline{K}{}^0_{\!R}(M)$ in~\cite{Lsec}. Here, $R$ is a ring satisfying a few technical assumptions with a representation~$\rho: R \to \operatorname{End} \C^{n}$, and all flat vector bundles arise from local systems of~$R$-modules by tensoring with~$\C^{n}$. Similarly, relations come from short exact sequences of such local systems.

Let now~$p \colon E \to B$ be a proper submersion with fibre~$M$. We choose~$T^{H}E$ and~$g^{TM}$ as before. 

\begin{Definition}\label{Def3.3} Let~$(F,g^{F},\alpha)$ be a generator of~$\overline{K}{}^0_{\!R}(M)$ and let~$g_{L^2}^{H}$ denote the $L^{2}$-metric on the virtual vector bundle
\[ H = \bigoplus_{k=0}^{\dim M}(-1)^{k}H^{\bullet}(E/B;F) \to B. \]
Then the push-forward of~$(F,g^{F},\alpha)$ is defined as
\[ p_{!}(F,g^{F},\alpha) = \biggl(H,\;g_{L^2}^{H},\;\int_{E/B}e\bigl(TM,\nabla^{TM}\bigr)\,\alpha - {\mathcal T}\bigl(T^{H}E,g^{TM},g^{F}\bigr)\biggr). \]
\end{Definition}

Lott then verifies that~$p_{!}$ defines a push-forward map
\begin{equation}\label{eq3.6}
p_{!} \colon \overline{K}{}^0_{\!R}(E) \to \overline{K}{}^0_{\!R}(B). 
\end{equation}
Moreover, on the level of $K$-theory, the push-forward is independent of the choices of~$T^{H}E$ and~$g^{TM}$.

\begin{Theorem}[Bunke~\cite{Bufn}]\label{thm3.3a}
  Lott's secondary $K$-groups together with the pushforward
  define a functor from the category of smooth proper submersions
  to the category of abelian groups.
\end{Theorem}

The proof is based on Ma's Theorem~\ref{thm3.1}.
Bunke shows that if~$p_{1} \colon E \to B$ and~$p_{2} \colon D \to E$
are smooth proper submersions, then
\begin{equation}\label{eq3.7}
(p_{1} \circ p_{2})_{!} = p_{1!} \circ p_{2!} \colon \overline{K}{}_{\!R}^0(D) \to \overline{K}{}_{\!R}^0(B).
\end{equation}
A similar push-forward in secondary $L$-theory has been defined by Bunke
and Ma~\cite{BM}, correcting an older definition by Lott~\cite{Lsec}.

\subsection{Rigidity of Kamber-Tondeur classes}\label{sect3.3}
In this section,
we discuss the dependence of the Kamber-Tondeur forms and the torsion forms
on the flat structure on the bundle~$F$.
Let~$V \to M$ be a vector bundle and assume that~$(\nabla^{V,t})_{t \in [0,1]}$ is a family of flat connections on~$V$. If we define a connection~$\nabla^{\tilde{V}}$ on the pull-back~$\tilde{V}$ of~$V$ to~$M \times [0,1]$ such that~$\nabla^{\tilde{V}}|_{M \times \{t\}} = \nabla^{V,t}$, then the connection~$\nabla^{\tilde{V}}$ will in general not be flat. In particular, the arguments in~\eqref{eq1.3} and~\eqref{eq1.7} are not applicable here. If we fix a family of metrics~$(g_{t}^{V})$ on~$V$, we have a family~$(\nabla^{V,t,*})_{t \in [0,1]}$ of adjoint connections that are again flat. Let now~$\tilde{V}$ denote the pullback of~$V$ to~$M \times [0,1]^{2}$ and construct~$\nabla^{\tilde{V}}$ such that
\begin{equation}\label{eq3.8}
  \nabla^{\tilde{V}}\big|_{M \times \{s\}\times [0,1]}
  = (1-s)\nabla^{V,t} + s \nabla^{V,t,*}. 
\end{equation}
We define forms~$L\left((\nabla^{V,t},g_{t}^{V})_{t}\right) \in \Omega^{\mathrm{even}}(M)$ by
\begin{equation}\label{3.9}
  L\left((\nabla^{V,t},g_{t}^{V})_t\right)
  = \pi i \int _{0}^{1} \int _{0}^{1}\iota_{\tfrac{\partial}{\partial s}}
    \iota_{\tfrac{\partial}{\partial t}}
    \ch\bigl(\tilde{V},\nabla^{\tilde{V}}\bigr)\, dt\, ds.
\end{equation}
Because~$\nabla^{V,t}$ and~$\nabla^{V,t,*}$ are flat, for~$s \in \{0,1\}$, we have
\begin{equation}\label{eq3.10}
  \ch\bigl(\tilde{V},\nabla^{\tilde{V}}\bigr)\bigr|_{M \times \{0,1\}\times [0,1]}
  = \begin{cases}
    \frac12\tr_V\left(\frac{\partial}{\partial t} \nabla^{V,t}\right)\, dt
      & s = 0, \\
    \noalign{\smallskip}
    \frac12\tr_V\left(\frac{\partial}{\partial t}\nabla^{V,t,*}\right)\, dt
      & s = 1.
  \end{cases}
\end{equation}
Hence it follows from Stokes' theorem that
\begin{equation}\label{eq3.11}
  dL\left((\nabla^{V,t})_{t},g^{V}\right)^{[\ge 2]}
  = \cho\left(V_{1},g_{1}^{V}\right)^{[\ge 3]}
    - \cho\left(V_{0},g_{0}^{V}\right)^{[\ge 3]},
\end{equation}
where~$V_{t}$ denotes the flat vector bundle~$(V,\nabla^{V,t})$.

One can show that~$L((\nabla^{V,t},g_{t}^{V})_{t})$ changes by exact forms if one replaces~$(\nabla^{V,t})_{t}$ by a homotopic path of flat connections. On the other hand, if~$\nabla^{V,1} = \nabla^{V,0}$, then the cohomology class of~$L((\nabla^{V,t},g_{t}^{V})_{t})$ depends on the homotopy class of the loop~$(\nabla^{V,t})_{t}$ in the space of flat connections.

Now assume that~$V$ is $\Z$-graded and that~$(V^{\bullet},\nabla^{V,t} + a_{t})$ is a parallel family of cochain complexes on~$M$ such that the fibrewise cohomology~$H^{\bullet}(V,a_{t})$ has the same rank for all~$t$. Then we obtain a family of flat Gau\ss-Manin connections~$(\nabla^{H,t})_{t}$ and a family of metrics~$g_{V,t}^{H}$ on a fixed vector bundle~$H \to M$. 

\begin{Theorem}[Rigidity, Bismut and G.~\cite{BGmorse}]\label{thm3.4} Under
these assumptions,
\begin{multline}\label{eq3.12}
T\left(\nabla^{V,1}+a_{1},g^{V}\right)^{[\ge 2]} - T\left(\nabla^{V,0} + a_{0},g^{V}\right)^{[\ge 2]}\\
 = L\left((\nabla^{V,t},g_{t}^{V})_{t}\right)^{[\ge 2]} - L\bigl((\nabla^{H,t}, g_{V,t}^{H})_{t}\bigr)^{[\ge 2]}.
\end{multline}
\end{Theorem}

Similarly, let~$(\nabla^{F,t})_{t}$ be a family of flat connections on~$F \to E$ such that the fibrewise cohomology~$H^{\bullet}(E/B;F_{t})$ has the same rank for all~$t$. Then we again have a family~$(H_{t},g_{L^{2},t}^{H})$ of flat vector bundles over~$B$.

\begin{Theorem}[Rigidity, Bismut and G.~\cite{BGmorse}]\label{thm3.5} Under these assumptions,
\begin{multline}\label{eq3.13}
{\mathcal T}\bigl(T^{H}E,g^{TM},g^{F_1}\bigr)^{[\ge 2]} - {\mathcal T}\bigl(T^{H}E,g^{TM},g^{F_0}\bigr)^{[\ge 2]} \\
= \int _{M/B}e\bigl(TM,\nabla^{TM}\bigr)\,L\left((\nabla^{F,t},g_{t}^{F})_{t}\right)^{[\ge 2]} - L\bigl((\nabla^{H,t},g_{L^{2},t}^{H})_{t}\bigr)^{[\ge 2]}.
\end{multline}
\end{Theorem}

Because~$T(T^{H}E,g^{TM},g^{F})^{[0]}$ equals the Ray-Singer analytic torsion, we cannot expect Theorems~\ref{thm3.4} and~\ref{thm3.5} to hold for the scalar part of the Bismut-Lott torsion, too. In fact these theorems as well as the construction of~${\mathcal T}(E/B;F)$ indicate that the ``higher'' Bismut-Lott torsion has a different topological meaning than the Ray-Singer torsion.

\subsection{Equivariant analytic torsions}\label{sect3.4}
Let us assume that~$p \colon E \to B$ is associated to a $G$-principal bundle~$P \to B$ for some compact, connected Lie group~$G$. In particular, $G$ acts by isometries on the fibre~$(M,g^{TM})$. Let~$F \to M$ be a $G$-equivariant flat vector bundle such that elements~$X$ of the Lie algebra~$g$ of~$G$ act by~$\nabla^{F}_{X^{M}}$, where~$X^{M}$ is the corresponding Killing field on~$M$. Then the induced vector bundle
\begin{equation}\label{eq3.14}
P\times_{G}F \longrightarrow E = P\times_{G}M,
\end{equation}
which we will again call~$F$, is also flat. A $G$-equivariant fibre bundle connection~$T^{H}P$ defines~$T^{H}E$, and we also fix a $G$-invariant metric on~$F$. Let~$\Omega \in \Omega^{2}(B;{\mathfrak g})$ denote the curvature of~$T^HP$.

It was already observed in~\cite{BL} and~\cite{Lgr} that in this situation, the Bismut-Lott torsion is given by an Ad-invariant formal power series~${\mathcal T}_{{\mathfrak g}}(g^{TM},g^{F}) \in \C[\![ \mathfrak g^{*}]\!]$ on~${\mathfrak g}$, such that
\begin{equation}\label{eq3.15}
{\mathcal T}\bigl(T^{H}E,g^{TM},g^{F}\bigr) = {\mathcal T}_{\tfrac{\Omega}{2\pi i}}\bigl(g^{TM},g^{F}\bigr) \in \Omega^{\mathrm{even}}(B).
\end{equation}

On the other hand, there is a $G$-equivariant generalisation of the Ray-Singer analytic torsion. If~$g \in G$ acts by isometries on~$M$ and preserves~$\nabla^{F}$, put
\begin{equation}\label{eq3.16}
\begin{aligned}
\vartheta_{g}\bigl(g^{TM},g^{F}\bigr)(s) &= -\str\left(N^{M}g(d_{M}+d_{M}^{*})^{-2s}\right)\\
\mbox{and} \quad {\mathcal T}_{g}\bigl(g^{TM},g^{F}\bigr) &= \frac{\partial}{\partial s} \vartheta_{g}\bigl(g^{TM},g^{F}\bigr)(s).
\end{aligned}
\end{equation}

Inspired by results of Bismut, Berline and Vergne about the equality of two notions of the equivariant index~\cite{BGV}, one can ask if the infinitesimal equivariant Bismut-Lott torsion~${\mathcal T}_{\mathfrak g}(g^{TM},g^{F})$ is related to the equivariant torsion~${\mathcal T}_{G}(g^{TM},g^{F})$.
Bunke proved in~\cite{Bugr} and~\cite{Buinf} that both equivariant torsions can be computed from the $G$-equivariant Euler characteristic of~$M$ up to a constant when~$G$ is connected and~$F$ satisfies some technical assumptions. From Bunke's results, one can deduce a relation between both equivariant torsions in some interesting special cases.

To state a more general relation between both equivariant torsions, we need the infinitesimal Euler form~$e_{\mathfrak g}(TM,\nabla^{TM}) \in \Omega^{\bullet}(M)[\![\mathfrak g^{*}]\!]$ and an equivariant Mathai-Quillen current~$\psi_{X}(TM,\nabla^{TM})$ on~$M$ such that
\begin{equation}\label{eq3.17}
d\psi_{X}\bigl(TM,\nabla^{TM}\bigr) = e_{X}\bigl(TM,\nabla^{TM}\bigr) - e\bigl(TM_{X},\nabla^{TM_X}\bigr)\delta_{M_X},
\end{equation}
where~$\delta_{M_X}$ is the Dirac current of integration over the fixpoint set~$M_{X}$ of the Killing field~$X^{M}$, for~$X \in {\mathfrak g}$. Finally, for a proper submersion~$p \colon E \to B$ with typical fibre~$M$ and a fibrewise $G$-action, there exists an even closed form~$V_{X}(E/S,T^{H}E, g^{TM})$ that is locally computable on~$E$, vanishes for even-dimensional fibres, and satisfies
\begin{equation}\label{eq3.18}
V_{rX}\bigl(E/S,T^{H}E, g^{TM}\bigr) = \frac{1}{|r|}r^{-\tfrac{N^{B}}{2}}V_{X}\bigl(E/S,T^{H}E,g^{TM}\bigr)
\end{equation}
for all~$r \in \R\backslash\{0\}$. In particular, the class~$V_{X}(E/S)\in H^{\mathrm{even}}(B;\R)$ is independent of~$T^{H}E$ and~$g^{TM}$. Let~$V_{X}(M) = V_{X}(E/S)^{[0]}$ denote the scalar part.

\begin{Theorem}[Bismut and G.~\cite{BGeq}]\label{thm3.6} For~$X \in {\mathfrak g}$, the equivariant torsions are related by
\begin{multline}\label{eq3.19}
{\mathcal T}_{X}\bigl(g^{TM},g^{F}\bigr) - {\mathcal T}_{e^{X}}\bigl(g^{TM},g^{F}\bigr)\\
 = \int_{M}\psi_{X}\bigl(TM,\nabla^{TM}\bigr)\,\cho\bigl(F,g^{F}\bigr) + V_{X}(M)\,\rk F\, .
\end{multline}
\end{Theorem}

Similar results for equivariant $\eta$-invariants
have been proved in~\cite{Geta}, and for equivariant holomorphic torsion
by Bismut and the author in~\cite{BGh}.

\begin{Remark}\label{rem3.6a}
In~${\mathcal T}_{e^{X}}(g^{TM},g^{F})$, all powers of~$X$ occur simultaneously. Thus, Theorem~\ref{thm3.6} can only hold for one choice of constants~$c_{k}$ in~\eqref{eq1.7c}, and this is precisely the so-called Chern normalisation introduced in~\cite{BGmorse} and also used in this overview.
The Chern normalisation is also needed for Lott's noncommutative
higher torsion classes in~\cite{Lnc}, see Remark~\ref{Rem7.6} below.
\end{Remark}

The theorem above is of course compatible with Bunke's computations. As a simple application, we can use K\"ohler's computation of the 
equivariant analytic torsion on compact symmetric spaces~\cite{Ksym} to compute the Bismut-Lott torsion of bundles with compact symmetric fibres and compact structure groups.
The case of sphere bundles will be important later.
Let~$\zeta$ denote the Riemann $\zeta$-function. We define an additive characteristic class~$\Jnull(W)$ for a vector bundle~$W \to M$ by
\begin{equation}\label{eq3.20}
\Jnull(W) = \frac{1}{2}\sum_{k=0}^{\infty}\zeta'(-2k)\,\ch(W)^{[4k]} \in H^{\bullet}(M;\R).
\end{equation}

\begin{Corollary}[Sphere bundles, Bunke~\cite{Buinf}, Bismut and G.~\cite{BGmorse}]\label{cor3.7}
Let~$E \to B$ be the unit $n$-sphere bundle of an oriented real vector bundle~$W \to B$. Then
\[ {\mathcal T}(E/B;\C) = \chi(S^{n})\,\Jnull(W). \]
\end{Corollary}

The meaning of the class~$V_{X}(M)$ is not quite clear from Theorem~\ref{thm3.6}. As Bismut explains in~\cite{Bhypo}, the Bismut-Lott torsion of a smooth proper submersion~$p \colon E \to B$ is formally given by evaluating~$V$ on the generator of the natural~$S^{1}$ action on the fibrewise free loop space~$L_{B}E$, viewed as a bundle over~$B$. Although the flat vector bundle~$F$ and its cohomology are not visible in this approach, many properties of~$V_{X}(E/S)$ proved in~\cite{BGeq} mirror well-known properties of Bismut-Lott torsion, including the behaviour under iterated fibrations in Section~\ref{sect3.1} and under Witten deformation in Section~\ref{sect5.1}.

\subsection{The Ma-Zhang subsignature operator}\label{sect3.5}

In section~\ref{sect1.1}, we have constructed Kamber-Tondeur forms by lifting the Chern character to flat vector bundles. In section~\ref{sect2.2}, we have constructed the torsion form as a correction term in a family index theorem. Thus, Bismut-Lott torsion is a double transgression of the Chern character. Ma and Zhang first produce an $\eta$-invariant, which can be regarded as a transgression of the Chern character. Then they derive Bismut-Lott torsion from a transgression of $\eta$-forms in~\cite{MZ1}. In other words, they get torsion forms by a different double transgression. On the way, they give a new analytic proof of Theorems~\ref{thm1.1} and~\ref{thm1.1a}.
Dai and Zhang have recently given a related construction in~\cite{DZ},
where Bismut-Lott torsion appears in the adiabatic limit of a Bismut-Freed
connection form that is related to Ma and Zhang's $\eta$-invariant.

Let~$p \colon E \to B$ be a proper submersion of closed manifolds, where~$B$ is oriented, and let~$F \to E$ be a flat vector bundle, then~$F$ is rationally trivial in the topological $K$-theory of~$E$. Thus there exists an isomorphism~$qF \cong E \times \C^{q \rk F}$ for some positive integer~$q$. Let~$\nabla^{0}$ denote the trivial flat connection on~$E \times \C^{q \rk F}$, then
\begin{equation}\label{eq8.4}
\widehat{\ch}\bigl(F,\nabla^{F}\bigr) = \frac{1}{q} \widetilde{\ch}\bigl(\nabla^{0},\nabla^{qF}\bigr) \in H^{\bullet}(E;\C/\Q).
\end{equation}

Choose~$T^{H}E$, $g^{TM}$, $g^{F}$ as before. We also choose a metric~$g^{TB}$ on~$B$ and put~$g^{TE} = g^{TM} \oplus p^{*}g^{TB}$ using the splitting~$TE = TM \oplus T^{H}E$. Let~$W \to B$ be a Hermitian vector bundle with metric~$g^{W}$ and connection~$\nabla^{W}$. Ma and Zhang consider two operators~$D_{\mathrm{sig}}^{W,F}$ and~$\hat{D}_{\mathrm{sig}}^{W,F}$ on~$\Omega^{\bullet}(E;p^{*}W \oplus F)$. Whereas~$D_{\mathrm{sig}}^{W,F}$ is an honest Dirac-operator if~$g^{F}$ is parallel, the operator~$\hat{D}_{\mathrm{sig}}^{W,F}$ differentiates only in the directions of the fibres. These operators should be viewed as ``quantisations'' in the sense of~\cite{BGV}, applied to the Bismut type superconnection~$\tilde{\A} = \frac{1}{2}(\tilde{\A}' + \tilde{\A}'')$ and the operator~$\tilde{\X}$ of~\eqref{eq2.21}.

If~$B$ is odd-dimensional, then
\begin{equation}\label{eq8.5}
D_{\mathrm{sig}}^{W,F}(r) = D_{\mathrm{sig}}^{W,F} + ir \hat{D}_{\mathrm{sig}}^{W,F}
\end{equation}
is a selfadjoint operator on~$\Omega^{\mathrm{even}}(B;\Omega^{\bullet}(E/B;p^{*}W \oplus F))$ for all~$r \in \R$.
The {\em reduced $\eta$-invariant} of~$D_{\mathrm{sig}}^{W,F}(r)$ is as usual defined as
\begin{equation}\label{eq8.6}
\overline{\eta}\left(D_{\mathrm{sig}}^{W,F}(r)\right) = \frac{1}{2}\left(\eta\bigl(D_{\mathrm{sig}}^{W,F}(r)\bigr) + \dim \ker \bigl(D_{\mathrm{sig}}^{W,F}(r)\bigr)\right) \in \R/\Z.
\end{equation}
For the virtual bundle~$H(E/B;F) \to B$, one defines similarly
\begin{equation}\label{eq8.7}
\overline{\eta}\left(D_{\mathrm{sig}}^{W,H}(r)\right) = \sum_{k}(-1)^{k}\,\overline{\eta}\left(D_{\mathrm{sig}}^{W,H^{k}(E/B;F)}(r)\right) \in \R/\Z.
\end{equation}

For~$\varepsilon > 0$, let~$D_{\mathrm{sig},\varepsilon}^{W,F}(r)$ denote the analogous operator, where the metric~$g^{TB}$ has been replaced by~$\frac{1}{\varepsilon}g^{TB}$.
The reduced $\eta$-invariants are related in the adiabatic limit~$\varepsilon \to 0$.

\begin{Theorem}[Ma and Zhang~\cite{Zss}, \cite{MZ1}]\label{thm8.2}
One has
\begin{equation}\label{eq8.8}
\lim_{\varepsilon \to 0} \overline{\eta}\left(D_{\mathrm{sig},\varepsilon}^{W,F}(r)\right) = \overline{\eta}\left(D_{\mathrm{sig}}^{W,H}(r)\right) \in \R/\Z\;.
\end{equation}
\end{Theorem}

\begin{proof}[Proof of Theorem~\ref{thm1.1a}] The proof for the imaginary part of~$\widehat{\ch}$ uses the identities
\begin{equation}\label{eq8.9}
\begin{aligned}
\frac{\partial}{\partial r}\bigg|_{r=0} \overline{\eta}\left(D_{\mathrm{sig},\varepsilon}^{W,F}(r)\right) &= \int_{B}L(TB)\, \ch(W)\, tr_{E/B}^{*}\sum_{k=0}^{\infty} c'_{k} \Im \widehat{\ch}(F)^{[2k+1]} \\
\mbox{and } \frac{\partial}{\partial r}\bigg|_{r=0} \overline{\eta} \left(D_{\mathrm{sig}}^{W,H}(r)\right) &= \int_{B}L(TB)\,\ch(W) \sum_{k=0}^{\infty} c'_{k}\Im \widehat{\ch}(H)^{[2k+1]}
\end{aligned}
\end{equation}
for some constants~$c'_{k} \ne 0$, where~$L(TB)$ denotes the Hirzebruch $L$-class. Because~$H^{\mathrm{even}}(B;\R)$ is spanned by the values of~$L(TB)\,\ch(W)$ for all complex vector bundles~$W$, one gets the imaginary part of~\eqref{eq1.16b} in Theorem~\ref{thm1.1a} from Theorem~\ref{thm8.2} by comparison of coefficients in~\eqref{eq8.9}.

The real part also follows from Theorem~\ref{thm8.2} because
\begin{equation}\label{eq8.10}
\overline{\eta}\bigl(D_{\mathrm{sig},\varepsilon}^{W,F}\bigr) - \rk F\, \overline{\eta}\left(D_{\mathrm{sig},\varepsilon}^{W}\right) = \int_{B}L(TB)\,\ch(W)\, tr_{E/B}^{*}\Re \widehat{\ch}(F) \in \R/\Q
\end{equation}
and a similar equation holds for the two virtual bundles~$H(E/B;F) \to B$ and~$\rk F \cdot H(E/B;\C) \to B$. To complete the proof, one needs that
\begin{equation}\label{eq8.11}
\widehat{\ch}(H(E/B;\C)) = 0 \in H^{\bullet}(B;\C/\Q)\;,
\end{equation}
which was already proved for evendimensional~$M$ by Bismut in~\cite{Bccs}.
\end{proof}

To recover Bismut-Lott torsion and Theorem~\ref{thm1.3} from this approach, one considers a generalised $\eta$-form
\begin{multline}\label{eq8.12}
\hat{\eta}_{r} = (2\pi i)^{-\frac{N^{B}+1}{2}}\int_{0}^{\infty}\Biggl(\tr_{s}\left(\left(\frac{\partial}{\partial t}\Bigl(\tilde{\A}_{t}+\frac{ir}{2}\tilde{\X}_{t}\Bigr)\right)e^{-\left(\tilde{\A}_{t} + \frac{ir}{2}\tilde{\X}_{t}\right)^{2}}\right)\\ - \frac{ir \cdot a}{\sqrt{1 + r^{2}}} t^{-\frac{3}{2}}\Biggr)\, dt
\end{multline}
for some locally computable function~$a \colon B \to \R$.

\begin{Theorem}[Ma and Zhang~\cite{MZ1}]\label{thm8.3} For certain constants~$c''_{k}\ne0$, one has
\begin{equation}\label{eq8.13}
\frac{\partial \hat{\eta}_{r}}{\partial r}\Bigr|_{r=0} = \sum_{k=0}^{\infty} c''_{k}\,d {\mathcal T}(T^{H}E, g^{TM},g^{F}).
\end{equation}
\end{Theorem}

Dai and Zhang will give a more explicit construction in~\cite{DZ}.
These last results seem to indicate a strong relation between Bismut-Lott torsion and $\eta$-forms that still has to be explored. A similar relation has been established by Braverman and Kappeler in a definition of complex-valued Ray-Singer torsion in~\cite{BK} for single manifolds.

\section{Igusa-Klein torsion}\label{sect4}
We have seen in sections \ref{sect1}--\ref{sect3} how to establish an index theorem for flat vector bundles using methods from local index theory for families, and how to discover Bismut-Lott torsion in a natural refinement of this index theorem. It is somewhat surprising that homotopy theoretical methods from differential topology lead to an invariant that is very closely related to Bismut-Lott torsion. There are several slightly different approaches to this topological higher torsion by Igusa and Klein~\cite{Kphd}, \cite{Ibuch}, \cite{IKpict}, \cite{IKfilt}.
In this section, we focus on Igusa-Klein torsion as described in~\cite{Ibuch}.
In section~\ref{sect6},
we discuss the approach by Dwyer, Weiss and Williams~\cite{DWW}.

\subsection{Generalised Morse functions and filtered complexes}\label{sect4.1}
It is well-known that smooth manifolds admit Morse functions. If~$p \colon E \to B$ is a smooth proper submersion, then in general, there is no function~$h \colon E \to \R$ that is a Morse function on every fibre of~$p$. However, by results of Igusa~\cite{Imorse} and Eliashberg and Mishachev~\cite{EM}, there always exist generalised Morse functions.

By a birth-death singularity of~$h \colon E \to \R$, we mean a fibrewise critical point of type A2 that is unfolded over~$B$. In other words, there exist~$k$,
a function~$h_0$ on~$B$
and coordinates~$u_1$, \dots\ on~$B$ and~$x_1$, \dots, $x_n$ along the fibres
such that locally,
	$$h(x,u)=h_0(u)+\frac{x_n^3}3-u_1x_n
	-\frac{x_1^2+\dots+x_k^2}2+\frac{x_{k+1}^2+\dots+x_{n-1}^2}2\;.$$
Birth-death singularities occur over a two-sided immersed submanifold~$B_{0} \subset B$ given by~$u_1=0$ in the coordinates above. Two fibrewise Morse critical points of adjacent indices over the ``positive'' side of~$B$ come together in a fibrewise cubical singularity. In a neighbourhood over the ``negative'' side, the function is regular.

Let~$C = C_{\mathrm M} \cup C_{\mathrm{bd}} \subset E$ denote the submanifold of fibrewise critical points of~$h$. Note that the submanifold~$C_{\mathrm M}$ of Morse fibrewise critical points of~$h$ locally covers~$B$, and that the submanifold~$C_{\mathrm{bd}}$ of birth-death critical points locally bounds two com\-ponents of~$C_{\mathrm M}$. After fixing a fibrewise metric~$g^{TM}$, the negative eigenspaces of the Hessian of~$h$ form a vector bundle~$T^{u}M \subset TM|_{C_{\mathrm M}}$ over~$C_{\mathrm M}$, whose rank is given by the Morse index~$\operatorname{ind}h$. At the birth-death singularities~$C_{\mathrm{bd}}$, the natural extension of~$T^{u}M$ of the two adjacent components of~$C_{\mathrm M}$ differ by an oriented trivial line bundle, the ``cubical direction''.

\begin{Definition}\label{Def4.1} A {\em generalised fibrewise Morse function} on~$p \colon E \to B$ is a function~$h \colon E \to \R$ that has only Morse and birth-death type fibrewise singularities. A {\em framed function} is a generalised fibrewise Morse function together with trivialisations of~$T^{u}M$ over each connected component of~$C_{\mathrm M}$ that extend up to the boundary, such that the two frames at each point of~$C_{\mathrm{bd}}$ differ only by the preferred generator of the cubical direction.
\end{Definition}

\begin{Theorem}[Igusa~\cite{Imorse}] \label{thm4.2} Let~$p \colon E \to B$ be a smooth fibre bundle with typical fibre~$M$. If~$\dim M \ge \dim B$, there exists a framed function, and if~$\dim M > \dim B$, it is unique up to homotopy.
\end{Theorem}

Here, uniqueness up to homotopy means that if~$h_{0},h_{1} \colon E \to \R$ are two framed functions, then there exists a framed function~$h \colon E \times [0,1] \to \R$ that restricts to~$h_{j}$ at~$E \times \{j\}$ for~$j = 0$, $1$.

If the dimension of the fibres is too small to apply Theorem~\ref{thm4.2}, one can take cross products with manifolds of Euler number 1, for example~$\R P^{2n}$. One can check that this will not alter the torsion classes of Igusa and Klein that we are going to introduce, so that the following constructions are valid for fibre bundles of arbitrary dimensions.

\subsection{Filtered chain complexes and the Whitehead space}\label{sect4.2}
We assume that we are given a smooth fibre bundle~$p \colon E \to B$ and a proper framed function~$h \colon E \to \R$ with finitely many fibrewise critical points over small subsets of~$B$. Let~$F \to E$ be a flat vector bundle. 

Over a small open subset~$U \subset B$, one can use~$h$ to filter the singular chain complexes of the fibres over~$U$. The filtered chain complexes are quasiisomorphic to a filtered chain complex on the vector space
\begin{equation}\label{eq4.1}
V'_{x} = \bigoplus_{C \in C'_{M}|_{x}} F_{c}.
\end{equation}
Here~$C'_{M}|_{x}$ is a subset of~$C_{\mathrm M}|_{x}$, where some pairs of components of~$C_{\mathrm M}$ near birth-death singularities are omitted. Both the filtration and the quasiisomorphism are natural and unique up to contractible choice.   

Moreover, the two leaves of~$C_{\mathrm M}$ near a birth-death singularity generate a direct summand isomorphic to
\begin{equation}\label{eq4.2}
  \begin{CD}
    0@>>>F@>\id>>F@>>>0
  \end{CD}
\end{equation}
after applying another quasiisomorphism that is again unique up to contractible choice. Adding or deleting a subcomplex of the form~\eqref{eq4.2} is called an {\em elementary expansion} or {\em elementary collapse}.

Suppose now that the flat bundle~$F$ is fibrewise acyclic and comes with an $R$-structure for a suitable ring~$R$ as in Section~\ref{sect3.2} above. Also assume that the holonomy of~$F$ if contained in some group~$G \subset GL_{r}(R)$, with~$r = \rk F$. A typical choice would be~$R = M_{r}(\C)$ and~$G = U(r)$ with~$r \in {\mathbb N}$. In~\cite{Ibuch}, Igusa constructs a classifying space for acyclic locally filtered finite dimensional chain complexes over~$R$ with holonomy~$G$, up to filtered quasiisomorphisms and elementary expansions and collapses. This space is called the {\em acyclic Whitehead space}~$Wh^{h}(R,G)$.
We give a slightly more explicit description in Section~\ref{sect5.2}.

\begin{Theorem}[Igusa~\cite{Ibuch}]\label{thm4.3} 
Each generalised fibrewise Morse function~$h \colon E \to \R$ gives rise to a classifying map
\begin{equation}\label{eq4.3}
\xi_{h}(E/B;F) \colon B \longrightarrow  Wh^{h}(R,G)
\end{equation}
that is unique up to homotopy.
\end{Theorem}

Together with Theorem~\ref{thm4.2}, one can associate to a smooth fibre bundle~$p \colon E \to B$ as above and a flat, fibrewise acyclic vector bundle~$F \to E$ with $R$-structure and holonomy group~$G$ a unique homotopy class of maps
\begin{equation}\label{eq4.4}
\xi(E/B;F) = \xi_{h}(E/B;F) \colon B \longrightarrow Wh^{h}(R,G),
\end{equation}
by choosing~$h$ to be a framed function.

Assume that~$G$ preserves a Hermitian metric, in other words, that~$F$ carries a parallel metric.
Then Igusa constructs cohomology classes
\begin{equation}\label{eq4.5}
  \tau=\sum_{k=1}^{\infty}\tau_{2k}\qquad\text{with}\qquad
  \tau_{2k}\in H^{4k}(Wh^{h}(R,G);\R)
\end{equation}
that are related to the Kamber-Tondeur classes of Section~\ref{sect1.1}.
These classes are natural under pairs of compatible ring and group homomorphisms~$(R,G)\to(S,H)$.
In particular,
it is enough to construct them for~$R=M_r(\C)$ and~$G=U(n)$.

\begin{Definition}\label{Def4.4} The {\em Igusa-Klein torsion} of a smooth fibre bundle~$p \colon E \to B$ as above and a flat fibrewise acyclic vector bundle~$F \to E$ with a parallel Hermitian metric is defined as
\begin{equation}\label{eq4.5b}
  \tau(E/B;F) = \xi(E/B;F)^{*}\,\tau \in H^{4\bullet}(B;\R)\;.
\end{equation}
\end{Definition}

Igusa also explains how to define~$\xi(E/B;F)$ and~$\tau(E/B;F)$ if~$H^{\bullet}(E/B;F) \to B$ is a trivial bundle, or more generally, a globally filtered flat vector bundle such that the associated graded vector bundle is trivial.
In other words, the flat cohomology bundle~$H^\bullet(E/B;F)\to B$
is then given by a unipotent representation of~$\pi_1B$.

\begin{Remark}\label{Rem4.4a}
  The map~$\xi(E/B;F)$ of~\eqref{eq4.4} is a higher torsion invariant
  in its own right. In fact, most of the properties of~$\tau(E/B;F)$
  in the next subsection  already hold at the level of~$\xi(E/B;F)$.
  Moreover, $\xi(E/B;F)$ is well-defined even if~$F$ carries no parallel
  metric.
  However, the cohomology class~$\tau(E/B;F)$ makes it possible to compare
  Igusa-Klein torsion with Bismut-Lott torsion.
\end{Remark}

\subsection{Properties of the Igusa-Klein torsion}\label{sect4.3}
Assume that the fibre bundle~$p \colon E \to B$ arises by gluing two families~$p_{i} \colon E_{i} \to B$ for~$i = 1$, $2$ along their fibrewise boundary~$\partial_{B}E_{1} = \partial_{B}E_{2}$. Then there exist a framed function~$h \colon E \to \R$ such that~$h|_{E_1} \ge 0$ and~$h|_{E_2} \le 0$. Igusa proves in~\cite{Ibuch} that the corresponding classifying map~$\xi(E/B;F) \colon B \to Wh^{h}(R,G)$ ``splits'' in an appropriate sense, at least if~$R$ is a field and the cohomology bundles~$H^{\bullet}(E_{i}/B;F|_{E_i}) \to B$ are unipotent as above.

Let~$DE_{i} := E_{i} \cup E_{i} \to B$ denote the fibrewise double of~$E_{i}$, and let~$F_{i} \to DE_{i}$ denote the flat vector bundle induced by~$F|_{E_i}$. Then the splitting above has the following consequence, in a wording suggested by Bunke.

\begin{Theorem}[Additivity, Igusa~\cite{Iax}]\label{thm4.5}
If~$E = E_{1} \cup E_{2} \to B$ is as above and the bundles~$H^{\bullet}(E_{i}/B;F|_{E_i}) \to B$ are unipotent, then
\begin{equation}\label{eq4.6}
2\tau(E/B;F) = \tau(DE_{1}/B;F_{1}) + \tau(DE_{2}/B;F_{2}).
\end{equation}
\end{Theorem}

Suppose that~$p \colon E \to B$ is a $(2n-1)$-sphere bundle with structure group~$U(1)^{n} \subset O(2n)$. Then~$E$ is the fibrewise join of~$n$ circle bundles over~$B$.
The Igusa-Klein torsion of circle bundles has been computed explicitly in~\cite{Ibuch}, and Theorem~\ref{thm4.5} gives the Igusa-Klein torsion of~$p$. By the splitting principle for vector bundles and naturality of the Igusa-Klein torsion, one can now compute the higher torsion of all unit sphere bundles in Euclidean vector bundles. We use the same normalisation as for Bismut-Lott torsion.
Let~${}^0\!J$ denote the characteristic class defined in equation~\eqref{eq3.20}.

\begin{Theorem}[Sphere bundles, Igusa~\cite{Ibuch}]\label{thm4.6}
Let~$V \to B$ be an oriented Euclidean vector bundle with unit sphere bundle~$p \colon E \to B$. Then
\begin{equation}\label{eq4.7}
\tau(E/B;F) = 2\,{}^0\!J(V)\;.
\end{equation}
\end{Theorem}
Note that this agrees with the computations of the Bismut-Lott torsion in Corollary~\ref{cor3.7} if the fibres are odd-dimensional.

Assume that~$h \colon E \to \R$ is a generalised fibrewise Morse function for~$p \colon E \to B$ that is not framed.
Because at the birth-death singularities~$C_{\mathrm{bd}}$, the natural extensions of~$T^{u}M$ at the two adjacent components of~$C_{\mathrm M}$ are stably isomorphic,
we have a class~${}^0\!J(T^uM)\in H^\bullet(C;\R)$.
Let~$\hat p=p|_C$ and let~$C^j_{\mathrm M}$ denote the fibrewise Morse critical
points of Morse index~$j$,
then there exists a well-defined push-down map
\begin{equation}\label{eq4.8}
\hat{p}_{*}\alpha = \sum_{j=0}^{\dim M}(-1)^{j}(p|_{C_{\mathrm M}^{j}})_{*}\,(\alpha|_{C_{\mathrm M}^{j}})
\in H^{\bullet}(B)
\end{equation}
for all~$\alpha\in H^\bullet(C)$.
One can compute the Igusa-Klein torsion of~$p \colon E \to B$ using the classifying map~$\xi_{h}(E/B;F)$ even though~$h$ is not framed. 

\begin{Theorem}[Framing principle, Igusa~\cite{Ibuch}, \cite{Icpx}]\label{thm4.7}
In the situation above,
\begin{equation}\label{eq4.9}
\tau(E/B;F) = \xi_{h}(E/B;F)^{*}\tau -
2\,\hat p_*{}^0\!J(T^uM)\, \operatorname{rk}F.
\end{equation}
\end{Theorem}
As an example, suppose that~$p \colon E \to B$ is the fibrewise suspension of the unit sphere bundle in a vector bundle~$V \to B$. Then there exists a fibrewise Morse function with only two fibrewise critical points, and the unstable tangent bundle at the fibrewise maximums is isomorphic to the pullback of~$V$. In this case, Theorems~\ref{thm4.6} and~\ref{thm4.7} give the same Igusa-Klein torsion for~$E \to B$.

\section{Bismut-Lott = Igusa-Klein}\label{sect5}
The higher analytic torsion of Bismut and Lott and the higher Franz-Reidemeister torsion of Igusa and Klein are defined using rather different methods. Nevertheless, it was noticed that both torsions assign special values of the Bloch-Wigner dilogarithm to acyclic flat line bundles over circle bundles over~$S^{2}$~\cite{IKpict}, \cite{BL}. In this section, we describe two approaches to prove that both torsions agree. The first, due to Bismut and the author, is inspired both by the proof of a general Cheeger-M\"uller theorem in~\cite{BZ1}, \cite{BZ2} and by the constructions of Igusa-Klein torsion using Morse theory~\cite{Ibuch}, \cite{Kphd}. The second approach classifies all invariants of smooth fibre bundles satisfying two simple axioms~\cite{Iax}. It is also suitable to compare Igusa-Klein torsion with the Dwyer-Weiss-Williams construction in~\cite{DWW},
see~\cite{BDKW} and Theorem~\ref{BDKWthm} below.
We also give some consequences of the equality of both torsions. 

\subsection{The Witten deformation}\label{sect5.1}
Let~$p \colon E \to B$ be a smooth proper submersion, and let~$F \to E$ be a flat vector bundle. We assume that there exists a fibrewise Morse function~$h \colon E \to \R$ such that the fibrewise gradient field~$\nabla h$ satisfies the Thom-Smale transversality condition on every fibre of~$p$. A nontrivial example is given by the fibrewise suspension of a unit sphere bundle at the end of section~\ref{sect4.3}.

Let~$o(T^{u}M) \to C$ denote the orientation bundle of~$T^{u}M \to C_{\mathrm M}$, which extends naturally to the birth-death singularities. Recall that~$C_{\mathrm M} = \bigcup_{j}C_{\mathrm M}^{j}$, where~$C_{\mathrm M}^{j}$ is the set of fibrewise critical points of Morse index~$j$. We define a finite-dimensional $\Z$-graded vector bundle
\[ V = \bigoplus_{j}V^{j} \longrightarrow B, \]
with
\begin{equation}\label{eq5.1}
V^{j} = \bigl(p|_{C_{\mathrm M}^{j}}\bigr)_{*}(F \otimes o(T^{u}M)).
\end{equation}
This bundle carries a flat connection~$\nabla^{V}$ induced by~$\nabla^{F}$, and a fibrewise Thom-Smale differential~$a$. Then~$a$ is parallel, so by~\cite{BL}, there exists a torsion form
\begin{equation}\label{eq5.2}
T\bigl(\nabla^{V}+a,g^{V}\bigr) \in H^{\bullet}(B;\R)
\end{equation}
as in Theorem~\ref{thm2.1} for all metrics~$g^{V}$ induced by metrics~$g^{F}$ on~$F$.

We choose a horizontal subbundle~$T^{H}E$ and a fibrewise Riemannian metric~$g^{TM}$ as in sections~\ref{sect1.3} and~\ref{sect2.2}. Then there exists a Mathai-Quillen current~$\psi(\nabla^{TM},g^{TM})$ on the total space of~$TM$, such that
\begin{equation}\label{eq5.3}
d\left((\nabla h)^{*}\psi(\nabla^{TM},g^{TM})\right) = e\bigl(TM,\nabla^{TM}\bigr) - \delta_{C},
\end{equation}
where~$\delta_{C}$ denotes the alternating sum of the currents of integration over~$C_{\mathrm M}^{j}$.

Recall that we have defined two metrics~$g_{V}^{H}$ and~$g_{L^{2}}^{H}$ on the flat vector bundle
\begin{equation}\label{eq5.4}
H = H^{\bullet}(M/B;F)\cong H^{\bullet}(V,a) \to B.
\end{equation}

\begin{Theorem}[Bismut and G. \cite{BGmorse}]\label{thm5.1}
Modulo exact forms on~$B$,
\begin{multline}\label{eq5.5}
{\mathcal T}\bigl(T^{H}M,g^{TM},g^{F}\bigr) = T\bigl(\nabla^{V}+a,g^{V}\bigr) + \chosl\bigl(H,g_{L^2}^{H},g_{V}^{H}\bigr) \\
+ \int_{M/B}(\nabla h)^{*}\psi\bigl(\nabla^{TM},g^{TM}\bigr) \cdot \cho\bigl(F,g^{F}\bigr) + \hat{p}_{*}\Jnull(T^{s}M-T^{u}M) \, \rk F.
\end{multline}
\end{Theorem}

This theorem is proved using the Witten deformation of the fibrewise de Rham complex by~$h$ as in~\cite{BZ1}, \cite{BZ2}. By Theorems~\ref{thm1.3} and~\ref{thm2.1} and~\eqref{eq1.7} and~\eqref{eq5.3}, taking the exterior derivative in Theorem~\ref{thm5.1} gives a trivial identity. The first three terms on the right hand side can be guessed that way. On the other hand, the last term contains topological information related to Igusa's framing principle.

In fact, if~$F$ carries a parallel metric, then~$T(\nabla^{V}+a,g^{V})^{[\ge 2]} = 0$ by the axiomatic description of~$T$ in~\cite{BL}, and the metric~$g_{V}^{H}$ is parallel, too. Recall the Becker-Gottlieb transfer~$tr_{E/B}^* \colon H^{\bullet}(E) \to H^{\bullet}(B)$ of~\eqref{eq1.11}. In this case, Theorem~\ref{thm5.1} reduces to
\begin{equation}\label{eq5.6}
\begin{aligned}
{\mathcal T}(E/B;F) &= \hat{p}_{*}\Jnull(T^{s}M - T^{u}M) \,\rk F \\
&= \hat{p}_{*}\Jnull(TM|_{C})\,\rk F - 2\hat{p}_{*}\Jnull(T^{u}M)\,\rk F \\
&= \tau(E/B;F) + tr_{E/B}^* \Jnull(TM)\,\rk F,
\end{aligned}
\end{equation}
where we have used a families version of the Poincar\'{e}-Hopf theorem, the framing principle of Theorem~\ref{thm4.7}, and the triviality of the classifying map~$\xi_{h}(E/B;F) \colon\penalty0 B \to Wh^{h}(R,G)$. This already explains the similarity of Corollary~\ref{cor3.7} and Theorem~\ref{thm4.6} for suspended unit sphere bundles.

\subsection{Analytic Igusa-Klein torsion}\label{sect5.2}
Let us assume again that~$h \colon E \to \R$ is a fibrewise Morse function. We still consider the~$\Z$-graded flat vector bundle~$V \to B$ of~\ref{sect5.1} with connection~$\nabla^{V}$ and metric~$g^{V}$ induced from~$F$. The function~$h$ acts by multiplication on~$F|_{C}$, giving rise to a selfadjoint endomorphism~$h$ of~$V$. An endomorphism of~$V$ is called $h${\em-upper triangular} if it maps each~$\lambda$-eigenvector of~$h$ to the sum of the $\mu$-eigenspaces with~$\mu > \lambda$.

For a generic fibrewise Riemannian metric~$g^{TM}$, the fibrewise gradient~$\nabla h$ will satisfy the Smale transversality condition over an open dense subset of~$B$. Over this subset, the Thom-Smale cochain differential is a parallel, $h$-upper triangular endomorphism~$a$ of~$V$. The various differentials~$a$ over different points along a path in~$B$ are conjugated by endomorphisms of~$V$ of the type~$\id + b$, where~$b$ is again $h$-upper triangular. As one moves around in a small circle on~$B$, these endomorphisms compose to an automorphism of~$(V,a)$ that is homotopic to the identity by an $h$-upper triangular homotopy. These various homotopies are again related by $h$-upper triangular higher homotopies, and so on. If the cohomology bundle~$H$ is unipotent, then all these structures are encoded in Igusa's map~$\xi_{h}(E/B;F) \colon B \to Wh^{h}(R,G)$ of Theorem~\ref{thm4.3}. 

One may also consider these algebraic structures as a singular superconnection on~$V$. If~$R = M_{r}(\R)$ or~$R = M_{r}(\C)$, there exists a smooth flat superconnection
\begin{equation}\label{eq5.7}
A' = \nabla^{V} + a_{0} + a_{1} + \dots 
\end{equation}
of total degree one with $h$-upper triangular
\begin{equation}\label{eq5.8}
a_{j} \in \Omega^{j}\bigl(B;\operatorname{End}^{1-j}(V)\bigr)
= \Omega^{j}\Bigl(B;\bigoplus_{k}\operatorname{Hom}\bigl(V^{k},V^{k+1-j}\bigr)\Bigr)
\end{equation}
and an $\Omega^\bullet(B)$-linear quasiisomorphism
\begin{equation}\label{eq5.9}
I\colon\bigl(\Omega^{\bullet}(E;F),\nabla^{F}\bigr) \to (\Omega^{\bullet}(B;V),A')
\end{equation}
by~\cite{G1}. The map~$I$ arises as a modification of the classical ``integration over the unstable cells'', and it maps forms supported on~$h^{-1}(\lambda,\infty)$ to the sum of the $\mu$-eigenspaces of~$h \in \operatorname{End} V$ with~$\mu \ge \lambda$. Moreover, the pair~$(A',I)$ is uniquely determined up to contractible choice by~$h$ and~$g^{TM}$. It is shown in~\cite{G2} that for acyclic~$F$, Igusa's map~$\xi_{h}(E/B;F)\colon B\to Wh^h(R,G)$ also classifies~$(A',I)$ up to a natural notion of homotopy.

The finite-dimensional torsion form of Definition~\ref{Def2.2} is only well-defined for flat superconnections of the form~$\nabla^{V}+a_{0}$. In~\cite{G1}, \cite{G2} a torsion form~$T(A',\nabla^{V},g^{V})$ is constructed using the fact that~$A'-\nabla^{V}$ is a form on~$B$ with values in a nilpotent subalgebra of End~$V$, which may vary over~$B$. We can still construct a metric~$g_{V}^{H}$ on
\begin{equation}\label{eq5.10}
H = H^{\bullet}(E/B;F) = H^{\bullet}(V,a_{0}) \to B 
\end{equation}
as in section~\ref{sect2.1}. Then we still have
\begin{equation}\label{eq5.11}
dT\bigl(A',\nabla^{V},g^{V}) = \cho\bigl(V,g^{V}) - \cho\bigl(H,g_{V}^{H}).
\end{equation}

\begin{Theorem}[\cite{G1}, \cite{G2}]\label{thm5.2}
Modulo exact forms on~$B$,
\begin{multline}\label{eq5.12}
{\mathcal T}\bigl(T^{H}E,g^{TM},g^{F}\bigr) = T\bigl(A',\nabla^{V},g^{V}\bigr) + \cho\bigl(H,g_{L^2}^{H},g_{V}^{H}\bigr) \\
+ \int_{E/B}(\nabla h)^{*}\psi\bigl(\nabla^{TM},g^{TM}\bigr) \cho\bigl(F,g^{F}\bigr) + \hat{p}_{*}\Jnull(T^{u}M-T^{s}M) \,\rk F.
\end{multline}
\end{Theorem}

If both~$F$ and~$H$ carry parallel metrics, we can construct a cohomology class as in Definition~\ref{Def2.8}. Let~$g^{V}$ be the induced parallel metric on~$V$.

\begin{Definition}\label{Def5.3} The analytic Igusa-Klein torsion is defined as
\begin{equation}\label{eq5.13}
T(E/B;F) = T\bigl(A',\nabla^{V},g^{V}\bigr)^{[\ge 2]} + \chosl\bigl(H,g^{H},g^{H}_{V}\bigr) \in H^{\mathrm{even},\ge 2}(B;\R).
\end{equation}
\end{Definition}

To justify the name, assume that~$g^{F}$ is parallel and the bundle~$H \to B$ is a trivial flat bundle. Then both~$\tau(E/B;F)$ and~$T(E/B;F)$ are defined.

\begin{Theorem}[\cite{G2}]\label{thm5.4}Under these assumptions,
\begin{equation}\label{eq5.14}
T(E/B;F) = \xi_{h}(E/B;F)^{*}\,\tau \in H^{4\bullet,\ge 4}(B;\R).
\end{equation}
\end{Theorem}

In contrast to the situation in the previous section~\ref{sect5.1}, these cohomology classes will be nontrivial in general. Also note that~$T(A',\nabla^{V},g^{V})$ can still be constructed for generalised fibrewise Morse functions~$h$ as in Definition~\ref{Def4.1}. In this context, Theorem~\ref{thm5.4} still holds.
A generalisation of Theorem~\ref{thm5.2} will be proved in~\cite{G3}.
As in~\eqref{eq5.6}, one can now compare Bismut-Lott torsion and Igusa-Klein torsion.

\begin{Theorem}[\cite{G2}, \cite{G3}] \label{thm5.5}
If~$F$ carries a parallel metric and~$H \to B$ is a trivial flat bundle, then
\begin{equation}\label{eq5.15}
{\mathcal T}(E/B;F) = \tau(E/B;F) + tr_{E/B}^* \Jnull(TM)\,\rk F.
\end{equation}
\end{Theorem}

\subsection{Axioms for higher torsions}\label{sect5.3}

In this section, we consider all smooth proper submersions~$p \colon E \to B$ with oriented fibres, such that the flat cohomology bundle~$H^{\bullet}(E/B;\C) \to B$ is unipotent in the sense of sections~\ref{sect4.2}, \ref{sect4.3}. We will consider characteristic classes~$\tau(E/B) \in H^{\bullet}(B;\R)$ of such fibre bundles that are natural under pullback. Such a class is called {\em additive} if it satisfies a gluing formula as in Theorem~\ref{thm4.5}.

Let~$W \to E$ be an oriented real vector bundle of rank~$n+1$, and let~$S \to E$ be its unit~$n$ sphere bundle. Then~$H^{\bullet}(S/B;\C) \to B$ is still unipotent. A characteristic class~$\tau$ as above is said to satisfy the {\em transfer relation\/} if
\begin{equation}\label{eq5.16}
\tau(S/B) = \chi(S^{n})\,\tau(E/B) + tr_{E/B}^*\tau(S/E) \in H^{\bullet}(B;\R).
\end{equation}
For the analytic torsion, the analogous result is a special case of Ma's transfer theorem~\ref{thm3.1}.

\begin{Definition}\label{Def5.6} A {\em higher torsion invariant} in degree~$k$ is a characteristic class~$\tau_{k}(E/B) \in H^{k}(B;\R)$ for all~$p \colon E \to B$ as above that is natural under pullback, additive, and satisfies the transfer relation~\eqref{eq5.16}.
\end{Definition}

\begin{Theorem}[Igusa~\cite{Iax}] \label{thm5.7} Higher torsion invariants exist in degree~$4k$ for all~$k > 0$, and every higher torsion invariant
is a linear combination of
\begin{gather*}
tr_{E/B}^*\Jnull(TM)^{[4k]}, \tag{even} \\
\mbox{and } \quad \tau_{2k}(E/B;\C) + tr_{E/B}^*\Jnull(TM)^{[4k]}. \tag{odd}
\end{gather*}
\end{Theorem}

Note that even higher torsion invariants vanish for~$p \colon E \to B$ if the fibres are odd-dimensional, and vice versa. The even classes~$tr_{E/B}^*\Jnull(TM)^{[4k]}$ are called {\em Miller-Morita-Mumford classes} in~\cite{Iax}, because they generalise the classes for surface bundles introduced in~\cite{Miller}, \cite{Morita}, \cite{Mum}. The odd higher torsion classes are multiples of the Bismut-Lott torsion~${\mathcal T}(E/B;\C)$ under the assumptions of Theorem~\ref{thm5.5}.

The following result follows from the proof of uniqueness in Theorem~\ref{thm5.7}.

\begin{Theorem}[Igusa~\cite{Iax}] \label{thm5.8} For fibre bundles~$p \colon E \to B$ as above,
  \begin{equation}\label{eq5.16a}
    \tau_{2k}(E/B;\C) \in H^{4k}(B;\zeta'(-2k)\Q)\;.
  \end{equation}
\end{Theorem}

Theorem~\ref{thm5.7} could in principle also
be used to prove Theorem~\ref{thm5.5}.
Unfortunately, additivity of the Bismut-Lott torsion is only known
as a consequence of Theorem~\ref{thm5.5}.
Another consequence of this result
is a more general transfer formula for Igusa-Klein torsion
as in Ma's Theorem~\ref{thm3.1},
including the case of fibre products.
By Theorems~\ref{thm3.6} and~\ref{thm5.5},
Igusa-Klein torsion is also related
to equivariant torsion in the case of fibre bundles
with compact structure groups.
Finally,
Theorems~\ref{thm3.5} and~\ref{thm5.5} describe the variation
of Igusa-Klein torsion under changes
of the flat bundle~$F\to E$.

We already mentioned the smooth Dwyer-Weiss-Williams torsion. Its definition is given in~\cite{DWW}, see section~\ref{sect6.2} below. In~\cite{BDW},
corresponding cohomology classes in~$H^{4k}(B;\R)$ are constructed.
Additivity and the transfer relation have recently been proved in~\cite{BDKW}.
This implies that cohomological smooth Dwyer-Weiss-Williams torsion shares
all the other properties mentioned above.
It also implies a more general transfer formula for Igusa-Klein torsion.

\section{Dwyer-Weiss-Williams torsion}\label{sect6}
In this section, we present the homotopy theoretical approach to generalised Euler characteristics and higher torsion invariants in~\cite{DWW} and~\cite{BDW}, and we sketch the proof of Theorem~\ref{thm1.2}.
Dwyer, Weiss and Williams construct three generalised Euler characteristics for fibrations~$p \colon E \to B$,
which contain information about the existence of
a topological or even smooth bundle of manifolds that is fibre homotopy
equivalent to~$p$.
If~$F\to E$ is a fibrewise acyclic bundle of $R$-modules,
then these Euler characteristics can be lifted to
three different higher torsion invariants.

\subsection{The topological index theorem}\label{sect6.1}
The Waldhausen $K$-theory~$A(E)$ of a space~$E$ is the $K$-theory of a certain category of retractive spaces over~$E$ \cite{Wkt}. It is a homotopy invariant functor, but not excisive, so it does not define a generalised homology theory. One can however define an excisive functor~$A^{\%}$ by putting
\begin{equation}\label{eq6.1}
A^{\%}(E) = \Omega^{\infty}(E_{+} \wedge A(*))\;.
\end{equation}
Here, $E_{+}$ is the disjoint union of~$E$ and a basepoint~$*$, and~$\Omega^{\infty}$ is the infinite loop space construction. Weiss and Williams construct a natural assembly map
\begin{equation}\label{eq6.2}
  \alpha: A^{\%}(E) \longrightarrow A(E)
\end{equation}
in~\cite{WWas}. We will also need the spectrum
\begin{equation}\label{eq6.3}
  Q(E_{+}) = \Omega^{\infty}\Sigma^{\infty}(E_{+})
  =\lim\nolimits_k\Omega^k\Sigma^k(E_{+})\;,
\end{equation}
where~$\Sigma$ denotes the reduced suspension.
For a fibration~$p \colon E \to B$, one has relative functors~$A_{B}(E) \to B$, $A_{B}^{\%}(E) \to B$ and~$Q_{B}(E_{B}) \to B$, which behave almost as fibrations over~$B$ where the functors above have been applied fibrewise to~$p \colon E \to B$. 

The {\em homotopy Euler characteristic}
\begin{equation}\label{eq6.4}
\chi^{h}(E/B) \colon B\longrightarrow A_{B}(E)
\end{equation}
is a section of~$A_{B}(E) \to B$. It is defined as the class of~$E \times S^{0}$ over~$E$ in~$A_{B}(E)$ if the fibres of~$p$ are {\em homotopy finitely dominated}, that is homotopy equivalent to retracts of finite CW complexes. If~$B$ is a point, then~$\chi^{h}(E)$ encodes precisely the Euler number and the Wall finiteness obstruction of the fibre. A flat vector bundle~$F \to E$, or more generally, a bundle of finitely generated projective $R$-modules for some ring~$R$, induces a map~$\lambda_{F} \colon A(E) \to K(R)$ induced by taking homology relative to~$E$ with coefficients in~$F$. For the proof of Theorem~\ref{thm1.2}, one uses that the composition of maps
\begin{equation}\label{eq6.5}
  \begin{CD}
    B@>\chi^{h}(E/B)>>A_B(E)@>>>A(E)@>\lambda_F>>K(R)
  \end{CD}
\end{equation}
classifies the fibrewise cohomology~$H(E/B;F)\to B$ as a virtual bundle
and thus gives the left hand side of~\eqref{eq1.13}.

If~$p \colon E \to B$ is a bundle of topological manifolds, there exists a vertical tangent  microbundle~$TM \to E$. It has an Euler class~$e(TM)$ with coefficients in~$A_{B}^{\%}(E)$.
Let~$\wp$ denote the generalised fibrewise Poincar\'{e} duality~\cite{DWW}.
Then one can define a {\em topological Euler characteristic\/}~$\chi^t$ of~$p$
with the property that
\begin{equation}\label{eq6.6}
\chi^t(E/B) = \wp\, e(TM) \colon B \longrightarrow A_{B}^{\%}(E).
\end{equation}
The fibrewise assembly of~\eqref{eq6.2} maps it to~$A_{B}(E)$. One has a Poincar\'{e}-Hopf type index theorem.

\begin{Theorem}[Dwyer, Weiss and Williams~\cite{DWW}]\label{thm6.1}
For a bundle~$p \colon E \to B$ of compact topological manifolds, the sections~$\chi^{h}(E/B)$ and~$\alpha\circ \chi^t(E/B)$ of~$A_{B}(E) \to B$ are homotopic by a preferred path of sections. 

Conversely,
if~$\chi^h(E/B)$ lifts to~$A_B^\%(E)$,
then~$p$ is fibre homotopy equivalent to a bundle
of compact topological manifolds.
\end{Theorem}

If the vertical tangent bundle~$TM \to E$ is a topological disc bundle, then~$p \colon E \to B$ is called a {\em regular manifold bundle}, which includes the important special case of a proper submersion. In this case, one can define the Becker Euler class~$b(TM)$ with coefficients in the sphere spectrum. Its fibrewise Poincar\'{e} dual gives the Becker-Gottlieb transfer, regarded as a section
\begin{equation}\label{eq6.7}
\chi^d(E/B)=tr_{E/B} = \wp\, b(TM) \colon B \longrightarrow Q_{B}(E_{B}).
\end{equation}
Even though Becker-Gottlieb transfer is already defined for fibrations with homotopy finitely dominated fibres, we can regard it as a third generalised Euler characteristic~$\chi^d$ for regular manifold bundles by~\eqref{eq6.7}. There is a natural unit map~$\eta \colon Q_{B}(E_{B}) \to A_{B}^{\%}(E)$, and we have another Poincar\'{e}-Hopf type index theorem.

\begin{Theorem}[Dwyer, Weiss and Williams~\cite{DWW}]\label{thm6.2}
For a bundle~$p \colon E\penalty-50 \to\penalty-50  B$ of closed regular topological manifolds, the sections~$\chi^t(E/B)$ and~$\eta\circ tr_{E/B}$ of~$A_{B}^{\%}(E)\penalty-50  \to\penalty-50  B$ are homotopic by a preferred path of sections.
\end{Theorem}

\begin{proof}[Proof of Theorem~\ref{thm1.2}]
We regard the homotopy class of maps~$E \to K(R)$ induced
by the finitely generated projective $R$-module bundle~$F\to E$.
As in~\eqref{eq6.5}, this map can be written as a composition
\begin{equation}\label{eq6.8}
  \begin{CD}
    E@>>>Q(E)@>\alpha\,\circ\,\eta>>A(E)@>\lambda_F>>K(R)\;.
  \end{CD}
\end{equation}
Thus, the right hand side of~\eqref{eq1.13} in Theorem~\ref{thm1.2} is classified by the composition
\begin{equation}\label{eq6.9}
  \begin{CD}
    B@>tr_{E/B}>>Q_{B}(E)@>\alpha\,\circ\,\eta>>A_{B}(E)@>>>A(E)@>\lambda_F>>K(R)\;.
  \end{CD}
\end{equation}
By Theorems~\ref{thm6.1} and~\ref{thm6.2}, this map is homotopic to~\eqref{eq6.5}, which classifies the left hand side of~\eqref{eq1.13}. This completes the proof.
\end{proof}

One notes that both sides of~\eqref{eq1.13} in Theorem~\ref{thm1.2} are defined for a fibration~$p \colon E \to B$ with homotopy finitely dominated fibres. However, for Theorem~\ref{thm6.2} one needs the regular structure coming from the smooth bundle structure. It is somewhat surprising that the existence of a smooth fibre bundle structure is necessary to compare the various Euler characteristics above.

\begin{Theorem}[Dwyer, Weiss and Williams~\cite{DWW}] \label{thm6.3}
Let~$p \colon E \to B$ be a fibration with homotopy finitely dominated fibres. If~$\chi^{h}(E/B)$ lifts to~$Q_B(E_B)\to B$,
then~$p$ is fibre homotopy equivalent to a bundle of smooth manifolds.
\end{Theorem}

\subsection{Topological higher Reidemeister torsion}\label{sect6.2}
Suppose that~$F \to E$ is a bundle of finitely generated projective $R$-modules that is fibrewise acyclic. Then the three Euler characteristics~$\chi^{h}(E/B)$,~$\chi^t(E/B)$ and~$tr_{E/B}$ of the previous subsection can be lifted to higher Reidemeister torsions.

Assume first that~$p \colon E \to B$ is a fibration with homotopy finitely dominated fibres. If~$F$ is fibrewise acyclic, then the composition in~\eqref{eq6.5} is canonically homotopic to the trivial map~$B \to K(R)$. For a single space~$M$, this gives an element~$\tau^{h}(M;F)$ in the homotopy fibre
\begin{equation}\label{eq6.10}
\Phi^{h}(M;F) = \operatorname{hofib}(\lambda_{F})
\end{equation}
of~$\lambda_{F}\colon  A(M)\to  K(R)$ over~$\chi^{h}(M) \in A(M)$.
For the fibration~$p$, we get a lift
\begin{equation}\label{eq6.11}
\tau^{h}(E/B;F) \colon B \longrightarrow \Phi^{h}(E/B;F) = \operatorname{hofib}\nolimits_{B}(\lambda_{F})
\end{equation}
of~$\chi^{h}(E/B)$, where the fibres of~$\Phi^{h}(E/B;F) \to B$ are the homotopy fibres of~$\lambda_{F}$.

If~$p \colon E \to B$ is a bundle of topological manifolds, we similarly get a lift of~$\chi^t(E/B): B \to A_{B}^{\%}(E)$ to
\begin{equation}\label{eq6.12}
\tau^{t}(E/B;F) \colon B \longrightarrow \Phi^{t}(E/B;F) = \operatorname{hofib}\nolimits_{B}(\lambda_{F}\circ\alpha).
\end{equation}
If~$p \colon E \to B$ is a bundle of smooth or regular manifolds, one gets a lift of~$\chi^d(E/B)=tr_{E/B} \colon B \to Q_{B}(E_{B})$ to
\begin{equation}\label{eq6.13}
\tau^{d}(E/B;F) \colon B \longrightarrow \Phi^{d}(E/B;F) = \operatorname{hofib}\nolimits_{B}(\lambda_{F}\circ\alpha\circ\eta).
\end{equation}

\begin{Definition}\label{Def6.4} If~$F \to E$ is fibrewise acyclic, then~$\tau^{h}(E/B;F)$, $\tau^{t}(E/B;F)$ and~$\tau^{d}(E/B;F)$ are called the {\em homotopy, topological} and {\em smooth Dwyer-Weiss-Williams torsion}, respectively, whenever they are defined.
\end{Definition}

The natural maps~$\alpha$ and~$\eta$ induce maps
\begin{equation}\label{eq6.14}
  \begin{aligned}
    \alpha \colon \Phi^{t}(E/B;F)
    &\longrightarrow \Phi^{h}(E/B;F)\\
    \text{and}\quad \eta \colon \Phi^{d}(E/B;F)
    &\longrightarrow \Phi^{t}(E/B;F).
  \end{aligned}
\end{equation}
By Theorems~\ref{thm6.1} and~\ref{thm6.2}, the Dwyer-Weiss-Williams torsions are related up to a preferred fibrewise homotopy by 
\begin{equation}\label{eq6.15}
\tau^{h}(E/B;F) \sim \alpha \tau^{t}(E/B;F)\quad\text{and}\quad \tau^{t}(E/B;F) \sim \eta \tau^{d}(E/B;F)
\end{equation}
if they are defined.

We will see in the next section~\ref{sect7} that Bismut-Lott torsion and Igusa-Klein torsion can detect different smooth bundle structures on a given topological manifold bundle~$p \colon E \to B$. Thus, ${\mathcal T}(E/B;F)$ and~$\tau(E/B;F)$ cannot be recovered from~$\tau^{h}(E/B;F)$ or~$\tau^{t}(E/B;F)$.
On the other hand,
we do not know any example yet where the difference~$\tau(E/B;F_1)-\tau(E/B;F_0)$ depends on the smooth fibre bundle structure if~$F_0$, $F_1\to E$ are two flat vector bundles of the same rank with unipotent fibrewise cohomology bundles.
It is thus natural to ask if one can recover~$\tau(E/B;F_1)-\tau(E/B;F_0)$
or~$\mathcal T(E/B;F_1)-\mathcal T(E/B;F_0)$
from~$\tau^{t}(E/B;F_1)-\tau^{t}(E/B;F_0)$ or even from~$\tau^{h}(E/B;F_1)-\tau^{h}(E/B;F_0)$.
Let us note at this point that additivity of the topological Dwyer-Weiss-Williams torsion~$\tau^{t}(E/B;F)$ and of the underlying Euler characteristic~$\chi^t(E/B)$ of~\eqref{eq6.6} has been established in~\cite{BDad}.

In~\cite{BDW}, a cohomological version of~$\tau^{d}(E/B;F)$
is constructed.
It is still defined if~$H^{\bullet}(E/B;F) \to B$ is a unipotent bundle.
The following result has recently been proved using Igusa's axioms.

\begin{Theorem}[Badzioch, Dorabia\l a, Klein, Williams~\cite{BDKW}]%
\label{BDKWthm}%
  For any~$k>0$, the cohomological smooth Dwyer-Weiss-Williams torsion
  of~\cite{BDW} is proportional to the Igusa-Klein torsion in the same degree.
\end{Theorem}

In addition, it would be nice to have a natural map from~$\Phi^{d}(E/B;F)$ to Igusa's Whitehead space~$Wh^{h}(R,G)$ that sends~$\tau^{d}(E/B;F)$ to the map~$\xi(E/B;F)$ of~\eqref{eq4.4}.

\section{Exotic smooth bundles}\label{sect7}

Consider two smooth proper submersions~$p_{i} \colon E_{i} \to B$ for~$i = 0$, $1$. It is possible that the fibres of~$p_{0}$ and~$p_{1}$ are diffeomorphic, and that there exists a homeomorphism~$\varphi \colon E_{0} \to E_{1}$ such that~$p_{0} = p_1\circ\varphi$, but no such diffeo\-mor\-phism. If this is the case, then~$p_{0}$ and~$p_{1}$ are isomorphic as topological, but not as smooth fibre bundles over~$B$. In this case, we will say that~$p_{1}$ gives an {\em exotic smooth bundle structure} on the bundle~$p_{0}$. Of course, in many cases there is no distinguished standard smooth bundle structure, so the term ``exotic'' may be misleading.
Higher torsion invariants detect some exotic smooth bundle structures, as we will explain in this section.
We also recall Heitsch-Lazarov torsion, which might be useful to
detect exotic smooth structures on foliations.

\subsection{Hatcher's example} \label{sect7.1}

It is well known that the higher stable homotopy groups of spheres are finite, whereas some higher homotopy groups of the orthogonal group are not. More precisely, if~$m$ is sufficiently large with respect to~$k$, then the kernel of the $J$-homomorphism
\begin{equation}\label{eq7.1}
J_{4k-1} \colon \pi_{4k-1}(O(m)) \longrightarrow \pi_{n+4k-1}(S^{m})
\end{equation}
contains an infinite cyclic subgroup. An element~$\gamma  \in \operatorname{ker} J_{4k-1}$ can be used to construct a family of embeddings~$\tilde{\gamma }_{q} \colon S^{m} \times D^{n-1} \to S^{m} \times D^{n-1}$ for~$q \in D^{4k}$, if~$n$ is sufficiently large, which are given by a pair of linear maps~$S^{m} \to S^{m}$ and~$D^{n-1} \to D^{n-1}$ for~$q \in S^{4k-1} = \partial D^{4k}$. Glueing~$D^{m+1} \times D^{n-1}$ to~$S^{m} \times D^{m}$ along~$S^{m} \times D^{n-1} \subset \partial(D^{m+1} \times D^{n-1})$ for all~$q \in D^{4k}$, one obtains an $(m+n)$-disc bundle over~$D^{4k}$ together with a canonical trivialisation over~$S^{4k-1}$. Thus, this bundle can be extended to a smooth disc bundle
\begin{equation}\label{eq7.2}
p_{\gamma } \colon E_{\gamma } \longrightarrow S^{4k} = D^{4k} \cup_{S^{4k-1}} D^{4k},
\end{equation}
as described in~\cite{Ibuch} and~\cite{G1}.

This disc bundle was first constructed by Hatcher. B\"{o}kstedt proved that for~$\gamma  \ne 0$, the bundle~$p_{\gamma }$ is homeomorphic, but not diffeomorphic to a trivial disc bundle in the sense above~\cite{Boek}. Note that~$p_{\gamma }$ carries a fibrewise Morse function~$h \colon E_{\gamma } \to \R$ with two critical points of index 0 and~$m$ in the part~$S^{m} \times D^{n}$ of the fibre, and another one of index~$m+1$ on~$D^{m+1} \times D^{n-1}$. The corresponding family of Thom-Smale complexes is trivial, but~$h$ is not framed. If~$W_{\gamma } \to S^{4k}$ denotes the $\R^{n}$-bundle with clutching function~$\gamma |_{S^{4k-1}}$, Igusa's framing principle gives
\begin{equation}\label{eq7.3}
  \tau(E_{\gamma }/S^{4k};\C) = 2(-1)^{m}\Jnull(W_{\gamma }) \ne 0 \in H^{4k}(S^{4k},\R),
\end{equation}
see Theorem~\ref{thm4.7} and~\cite{Ibuch}.

To construct a smooth proper submersion, we take the fibrewise double~$DE_{\gamma } \to S^{4k}$. Its Igusa-Klein torsion of Definition~\ref{Def4.4} is given by
\begin{equation}\label{eq7.4}
  \tau(DE_{\gamma }/S^{4k};\C) = 2\bigl((-1)^{m}-(-1)^{n}\bigr) \Jnull (W_{\gamma }),
\end{equation}
which vanishes precisely if the fibres are even-dimensional. By Theorem~\ref{thm5.5}, this agrees with the Bismut-Lott torsion~${\mathcal T}(DE_{\gamma }/S^{4k};\C)$ of Definition~\ref{Def2.8}, see~\cite{G1}.

If~$p \colon E \to B$ is a smooth proper submersion with~$\dim B = 4k$ and~$\dim M$ odd and sufficiently large, then one can take out~$r$ copies of~$D^{4k} \times D^{\dim M}$ from~$E$ and glue in~$r$ copies of~$E_{\gamma }|_{D^{4k}}$ instead. This gives an exotic smooth bundle~$p_{r} \colon E_{r} \to B$. If either Bismut-Lott torsion or Igusa-Klein torsion are defined for some flat bundle~$F \to E$, then this torsion will change by~$\pm 2r \Jnull(W_{\gamma })\,\rk F \in H^{4k}(B;\R)$ if~$B$ is oriented. Igusa also constructs a {\em difference torsion} satisfying
\begin{equation}\label{eq7.5}
\tau(E_{r}/B,E/B;F) = \pm 2r \Jnull(W_{\gamma }) \,\rk F
\end{equation}
even if~$H^{\bullet}(E/B;F) \cong H^{\bullet}(E_{r}/B;F) \to B$ is not a unipotent bundle. 

We still assume that~$B$ is oriented and that~$\dim M$ is odd and sufficiently large. The gluing construction above can be generalised to construct a discrete family of exotic smooth bundles~$p_{\nu} \colon E_{\nu} \to B$ such that the values of their difference torsions~$\tau(E_{r}/B,E/B;F)$ form a lattice in the space
\begin{equation}\label{eq7.6}
\bigoplus_{k=1}^{\infty} \wp\, \operatorname{im}\Bigl(p_{*}\colon H_{\dim B-4k}(E) \longrightarrow H_{\dim B-4k}(B)\Bigr) \subset \bigoplus_{k=1}^{\infty} H^{4k}(B)
\end{equation}
of classes that are Poincar\'{e} dual to classes pushed down from~$E$. This is an ongoing project with Igusa.

\subsection{The space of stable exotic smooth structures}\label{sect7.2}

There are two natural questions: can higher torsion detect all exotic smooth bundle structures, and can all these structures be constructed? To answer these questions, one wants to understand the space of all such exotic smooth bundle structures. As Williams pointed out, a certain stable version of this space can be analysed using the methods of the paper~\cite{DWW}.

We start with a bundle~$p \colon E \to B$ of compact topological $n$-manifolds, equipped with a vector bundle~$V \to E$ of rank~$n$. A smooth manifold bundle~$p' \colon E' \to B$ is called a {\em fibrewise tangential smoothing} of~$(E/B,V)$ if there exists a homeomorphism~$\varphi \colon E' \to E$ with~$p' = p \circ \varphi$ and a vector bundle isomorphism~$\operatorname{ker}(dp') \to V$ over~$\varphi$. Let~${\mathcal S}_{B}(E,V)$ denote the space of all fibrewise tangential smoothings. By considering total spaces of closed, even-dimensional linear disk bundles~$\pi: D(\xi)\subset \xi \to E$ after rounding off the corners, we construct the space of {\em stable fibrewise tangential smoothings}
\begin{equation}\label{eq7.7}
{\mathcal S}_{B}^{s}(E,V)
= \lim_{\longrightarrow} {\mathcal S}_{B}(D(\xi),\pi^{*}(V \oplus \xi))\;,
\end{equation}
where the limit is taken over all vector bundles.

Let~${\mathcal H}(*)$ be the stable $h$-cobordism space, and construct a fibration~${\mathcal H}_{B}^{\%}(E)$ with fibres~$\Omega^{\infty}(M_{+}\wedge {\mathcal H}(*))$ as in~\eqref{eq6.1}.

\begin{Theorem}[Dwyer, Weiss and Williams~\cite{DWW}]\label{thm7.1} If~$(E/B,V)$ admits stable fibrewise tangential smoothings, then~${\mathcal S}_{B}^{s}(E,V)$ is homotopy equivalent to the space of sections of~${\mathcal H}_{B}^{\%}(E) \to B$.
\end{Theorem}

In other words, the group~$\pi_{0}\Gamma_{B}{\mathcal H}_{B}^{\%}(E)$ of homotopy classes of sections acts simply transitively on the isomorphism classes of stable fibrewise tangential smoothings.

\begin{Theorem}[Igusa and G.]\label{thm7.2} If the fibres and base of~$p \colon E \to B$ are closed oriented manifolds, then
\begin{equation}\label{eq7.8}
\pi_{0}\bigl({\mathcal S}_{B}^{s}(E,V)\bigr) \otimes_{\Z}\Q \cong \bigoplus_{k=1}^{\infty} H_{\dim B-4k}(E;\Q).
\end{equation}
\end{Theorem}

In special cases, this was already known, see~\cite{FH}.
Thus, if~$p_{i} \colon D(\xi_{i}) \to B$ are stable fibrewise tangential smoothings for~$i = 0$, $1$, we can define the {\em relative Dwyer-Weiss-Williams torsion}~$\tau_{2k}^{d/t}(p_{0},p_{1}) \in H^{4k}(B;\Q)$ for~$k \ge 1$ as the Poincar\'{e} dual of the image of the corresponding difference class in~$H_{\dim B-4k}(B;\Q)$.

\begin{Theorem}[Igusa and G.]\label{thm7.3} In the situation above, the Igusa-Klein difference torsion is a scalar multiple of the relative Dwyer-Weiss-Williams torsion.
\end{Theorem}
Details will appear elsewhere.

\begin{Remark}\label{Rem7.4} In general, the space in~\eqref{eq7.8} has higher rank than the space in~\eqref{eq7.6}. This implies that higher torsion cannot detect all rational stable fibrewise tangential smoothings.
It does not help to chose different flat vector bundles~$F \to E$ either.
One reason is that $E$ could be simply connected.
Another reason is the fact that in~\eqref{eq7.5} and its analogue
in the more general setting of~\eqref{eq7.6},
the flat vector bundle~$F$ only contributes by its rank.
\end{Remark}

\begin{Remark}\label{Rem7.5}
Thus the difference of the Igusa-Klein
or Bismut-Lott torsions of~$E\to B$ with two different
flat vector bundles of the same rank
seems to be independent of the smooth structure in the examples known so far.
This observation leads to the question if this difference
can be computed already from the topological or the homotopy
Dwyer-Weiss-Williams torsion. Theorem~\ref{thm3.5} shows that under
special assumptions, it can even be computed using
the Becker-Gottlieb transfer only.
\end{Remark}

\begin{Remark}\label{Rem7.6}
In the special case of aspherical fibres~$M$, Lott defines a noncommutative higher analytic torsion form with coefficients in a certain subalgebra of~$C_{\mathrm r}^{*}\pi_{1}(M)$ in~\cite{Lnc}. Lott asks if this invariant detects all rational exotic structures. To the author's knowledge, this question is still open. More generally, one would like to have a similar invariant for arbitrary fibres that can detect all rational stable exotic smooth structures.
\end{Remark} 

\subsection{Heitsch-Lazarov torsion for foliations}\label{sect7.3}

Let~$E$ be a smooth closed manifold with a smooth foliation~${\mathcal F}$. Since in general, the space of leaves~$E/{\mathcal F}$ is ill-behaved, we consider a foliation groupoid whose elements are classes of paths on the leaves of~${\mathcal F}$. We will assume that this groupoid~${\mathcal G}$ lies between the homotopy and the holonomy groupoid, and that it is Hausdorff and thus given by a smooth manifold and two submersions~$r$, $s \colon {\mathcal G} \to E$. We will assume that the strong Novikov-Shubin invariants of the leafwise Hodge-Laplacians are positive.

Heitsch and Lazarov give a generalisation of Bismut-Lott torsion in a setting that essentially avoids noncommutative methods~\cite{HL}. Thus, let~$\Omega_{c}^{\bullet}(E/{\mathcal F})$ denote the H\"{a}fliger Forms, that is, the coinvariants under~${\mathcal F}$ in the space of compactly supported de Rham forms on a complete transversal to~${\mathcal F}$. The cohomology~$H_{c}^{\bullet}(E/{\mathcal F})$ of~$(\Omega_{c}^{\bullet}(E/{\mathcal F}),d)$ resembles the compactly supported de Rham cohomology of a manifold.

Let~$F \to E$ be a flat vector bundle with metric~$g^{F}$. If one fixes a complement~$T^{H}E$ to~$T{\mathcal F} \subset TE$ and a leafwise metric~$g^{T{\mathcal F}}$, there exists a natural connection~$\nabla^{T{\mathcal F}}$ on~$T{\mathcal F} \to E$. Using integration along the leaves, one defines
\begin{equation}\label{eq8.1}
\int_{{\mathcal F}}e\bigl(T{\mathcal F},\nabla^{T{\mathcal F}}\bigr)\,\cho\bigl(F,g^{F}\bigr) \in \Omega_{c}^{\bullet}(E/{\mathcal F}).
\end{equation}

Let~$P \colon \Omega^{\bullet}({\mathcal F};F) \to H^{\bullet} = H^{\bullet}({\mathcal F};F)$ denote the projection of the leafwise forms with values in~$F$ onto the harmonic forms. Using~$P$, one defines
\begin{equation}\label{eq8.2}
\cho\bigl(H,g_{L^2}^{H}\bigr) \in \Omega_{c}^{\bullet}(E/{\mathcal F})
\end{equation}
in analogy with~\eqref{eq1.5}.
As in Definition~\ref{Def2.4},
Heitsch and Lazarov then construct a higher analytic torsion form~${\mathcal T}(T^{H}E,g^{T{\mathcal F}},g^{F}) \in \Omega_{c}^{\bullet}(E/{\mathcal F})$.

\begin{Theorem}[Heitsch and Lazarov~\cite{HL}]\label{thm8.1} In the situation above,
\begin{equation}\label{eq8.3}
d{\mathcal T}\bigl(T^{H}E,g^{T{\mathcal F}},g^{F}\bigr) = \int_{{\mathcal F}}e\bigl(T{\mathcal F},\nabla^{T{\mathcal F}}\bigr)\cho\bigl(F,g^{F}\bigr) - \cho\bigl(H,g_{L^2}^{H}\bigr) \in \Omega_{c}^{\bullet}(E/{\mathcal F}).
\end{equation}
\end{Theorem}

Heitsch and Lazarov need a large positive lower bound for the strong leafwise Novikov-Shubin invariants. Thus they only regard examples with compact leaves. It seems however, that uniform positivity of the Novikov-Shubin invariants is sufficient to prove Theorem~\ref{thm8.1}. This is an ongoing joint project
with Azzali.

Given~${\mathcal F}$ as above with~$\dim {\mathcal F}$ odd and~$\dim E - \dim {\mathcal F} = 4k$, one can remove~$r$ disjoint foliated regions~$D^{4k} \times D^{\dim {\mathcal F}}$ and glue in~$r$ copies of the disc bundle~$E_{\gamma}|_{D^{4k}}$ of section~\ref{sect7.1}. It would be interesting to know if the Heitsch-Lazarov torsion~${\mathcal T}(T^{H}E,g^{T{\mathcal F}},g^{F})$ then changes by~$\pm 2r \Jnull(W_{\gamma})\,\rk F$ as in~\eqref{eq7.5}. In this case, we would have a new foliation~${\mathcal F}_{r}$ on~$E$ that is homeomorphic, but not diffeomorphic to the original foliation~${\mathcal F}$, and thus ``exotic''. Note that~${\mathcal F}$ and~${\mathcal F}_{r}$ have the same dynamics, since on a complete transversal that does not meet the modified regions, nothing changes. More generally, one would like to classify these exotic smooth structures, construct as many as possible explicitly, and see which of them can be distinguished by Heitsch-Lazarov torsion, or a noncommutative generalisation of it.

\section{The hypoelliptic Laplacian and Bismut-Lebeau torsion}\label{sect8}
In~\cite{Bhypo} and~\cite{BLprep}, Bismut and Lebeau consider an analytic torsion form that is defined using a hypoelliptic operator~$\mathfrak A_{b,\pm}^{2}$ on differential forms on the total space~$T^{*}M$ of the vertical cotangent bundle of the family~$p \colon E \to B$. While the fibrewise Hodge Laplacian generates a Brownian motion on the fibres~$M$ of~$p$, the operator~$\mathfrak A_{b,\pm}^{2}$ generates a stochastic version of the geodesic flow on~$T^{*}M$, where the velocities are perturbed by a Brownian motion for~$b \in (0,\infty)$. As~$b \to 0$, this process converges in an appropriate sense to the classical Brownian motion on~$M$. On the other hand, as~$b \to \infty$, one recovers the unperturbed geodesic flow. One motivation to study the family of operators~$(\mathfrak A_{b})$ is Fried's conjecture, which relates the torsion of a single manifold~$M$ to the closed orbits of a certain class of flows on~$M$, see~\cite{Fried} for an overview.

\subsection{The hypoelliptic Laplacian on the cotangent bundle}\label{sect8.1}
Let~$p \colon E \to B$ be a smooth proper submersion, and let~$F \to E$ be a flat vector bundle as before. Let~$\pi \colon T^{*}M \to E$ denote the vertical cotangent bundle. If one fixes~$T^{H}E \subset TE$ and~$g^{TM}$ as before, one obtains a splitting
\begin{equation}\label{eq8.14}
TT^{*}M \cong \pi^{*}(T^{H}E \oplus TM \oplus T^{*}M)
\end{equation}
and a corresponding splitting of the bundle~$\Omega^{\bullet}(T^{*}M/B;\pi^{*}F)$.
We regard the bundle~$\tilde E=E\times(0,\infty)^2\to\tilde B
=B\times(0,\infty)^2$,
and let~$(b,t)$ denote the coordinates of~$(0,\infty)^2$.

On the vertical part~$TM \oplus T^{*}M$ of~$TT^{*}M$, one defines a metric~$\mathfrak{g}$ by
\begin{equation}\label{eq8.15}
  \mathfrak{g} =
  \begin{pmatrix}
    \tfrac1t\,g^{TM}&\id_{T^{*}M}\\\id_{TM}&2t\,g^{T^{*}M}
  \end{pmatrix}\colon TM \oplus T^{*}M\longrightarrow\bigl(TM \oplus T^{*}M\bigr)^*\;.
\end{equation}
Together with a metric~$g^{F}$ on~$F \to E$ and the symplectic volume form on~$TM \oplus T^{*}M$, one obtains an $L^{2}$-metric~$\mathfrak{g}$ on the bundle~$\Omega^{\bullet}_{0}(T^{*}M/B;\pi^{*}F) \to B$ of compactly supported forms. On this bundle, there exists a $\mathfrak{g}$-isometric involution~$u$ with 
\begin{equation}\label{eq8.16}
  (u\alpha)_{(q,v)} =
  \begin{pmatrix}
    {\rm id}_{TM}&2t\,g^{T^{*}M}\\0&-{\rm id}_{T^{*}M}
  \end{pmatrix}\, \alpha_{(q,-v)}
\end{equation}
for all~$q \in E$, $v \in T_{q}^{*}M$, and thus, one can define a nondegenerate Hermitian form~$\mathfrak{h}$ of signature~$(\infty,\infty)$ by
\begin{equation}\label{eq8.17}
\mathfrak{h}(\alpha,\beta) = \mathfrak{g}(u\alpha,\beta).
\end{equation}
As before, let~$\mathbb A' = d_{E}$ denote the total exterior derivative on~$\Omega^{\bullet}(T^{*}M;\pi^{*}F)$, regarded as a superconnection on the bundle~$\Omega^{\bullet}_{0}(T^{*}M/B;\pi^{*}F)$. Define~$\bar{\mathbb A}'$ as the~$\mathfrak{h}$-adjoint of~${\mathbb A}'$. Again, ${\mathbb A}'$ and~$\bar{\mathbb A}'$ are flat superconnections.

One has the canonical one-form~$\vartheta \in \Omega^{1}(T^{*}M)$, with 
\begin{equation}\label{eq8.18}
d\vartheta = \omega^{H} + \omega^{V} \in \Gamma\bigl(\pi^{*}\bigl(\Lambda^{2}(T^{H}E)^{*} \oplus T^{*}M \otimes TM\bigr)\bigr) \subset \Omega^{2}(T^{*}M),
\end{equation}
where~$\omega^{V}$ is the standard symplectic form on the cotangent bundle of each fibre of~$p$. Consider the Hamiltonians
\begin{equation}\label{eq8.19}
\mathcal{H}_{\pm}(q,v) = \pm \frac{t^2}{2b^2}\,\|v\|_{T^{*}M}^{2}.
\end{equation}
Then for~$b=t=1$,
the $\omega^{V}$-gradient of~$\mathcal{H}_{+}$ is the generator
\begin{equation}\label{eq8.20}
{\rm sgrad} \mathcal{H}_{+}|_{(q,v,b,t)} = g^{T^{*}M}(v) \in T_{q}M
\end{equation}
of the geodesic flow on~$T^{*}M$ over the fibres~$M$ of~$p$.

Regard the flat superconnections
\begin{equation}\label{eq8.21}
\mathfrak{A}'_{\pm} = e^{-(\mathcal{H}_{\pm} - \omega^{H})} {\mathbb A}' e^{\mathcal{H}_{\pm} - \omega^{H}} \quad\text{and}\quad\bar{\mathfrak A}'_{\pm} = e^{\mathcal{H}_{\pm} - \omega^{H}}\bar{\mathbb A}' e^{-(\mathcal{H}_{\pm}-\omega^{H})}.
\end{equation}
Then there exists an $\mathfrak{h}$-selfadjoint superconnection~$\mathfrak{A}_{\pm}$ and an $\mathfrak{h}$-skew adjoint endomorphism~$\mathfrak{X}$ with
\begin{equation}\label{eq8.22}
\mathfrak{A}_{\pm} = \frac{1}{2}(\mathfrak{A}'_{\pm} +\bar{\mathfrak A}'_{\pm}) \quad\text{and}\quad \mathfrak{X}_{\pm} =\bar{\mathfrak A}'_{\pm} - \mathfrak{A}'_{\pm}.
\end{equation}
One finally defines
\begin{equation}\label{eq8.23}
  \mathfrak{A}_{b,t,\pm} = \mathfrak{A}_{\pm}|_{B\times\{(b,t)\}}
  \quad\text{and}\quad
  \mathfrak{X}_{b,t,\pm} = \mathfrak{X}_{\pm}|_{B\times\{(b,t)\}}\;.
\end{equation}
The operator~$\mathfrak{A}_{b,t,\pm}^{2} = -\mathfrak{X}_{b,t,\pm}^{2}$ is the sum of a harmonic oscillator along the fibres of~$\pi \colon T^{*}M \to E$, the Lie derivative by~$\operatorname{sgrad}\mathcal H_\pm$, and some terms of lower order or smaller growth at infinity. In particular, $\frac{\partial}{\partial u} - \mathfrak{A}^{2}_{b,t,\pm}$ is hypoelliptic in the sense of H\"{o}rmander, for an extra variable~$u \in {\mathbb R}$. By~\cite{BLprep}, the restriction~$\mathfrak{A}_{b,t,\pm}^{[0],2}$
of the operator~$\mathfrak{A}_{b,t,\pm}^{2}$ to the fibres of~$(p\circ\pi)\colon T^*M\to B$ has discrete spectrum and compact resolvent. Recall that~$n = \dim M$.

\begin{Theorem}[Bismut and Lebeau~\cite{BLprep}]\label{thm8.4}
The operator~$\mathfrak{A}_{b,t,\pm}^{\prime[0]}$ acts on the generalised 0-eigenspace~${\rm ker}(\mathfrak{A}_{b,t,\pm}^{[0],2N})$ of the operator~$\mathfrak{A}_{b,t,\pm}^{[0],2}$, where~$N \gg 0$, and for all~$k \in {\mathbb Z}$ and all~$b$, $t > 0$,
\begin{equation}\label{eq8.24}
\begin{aligned}
H^{k}({\rm ker}(\mathfrak{A}_{b,+}^{[0],2N}),\mathfrak{A}^{\prime[0]}_{b,+}) &\cong H_{+}^{k}(E/B;F) = H^{k}(E/B;F), \\
H^{k+n}({\rm ker}(\mathfrak{A}_{b,-}^{[0],2N}),\mathfrak{A}^{\prime[0]}_{b,-}) &\cong H_{-}^{k+n}(E/B;F) = H^{k}(E/B;F \otimes o(TM))\;.
\end{aligned}
\end{equation}
\end{Theorem}

Note that~$H_{-}^{n+k}(E/B;F^*)\cong H_{+}^{n-k}(E/B;F)$
by fibrewise Poincar\'e duality.

\subsection{Bismut-Lebeau torsion}\label{sect8.2}
We can now explain the higher torsion~$\mathcal T_{b,\pm}$
of the cotangent bundle defined by Bismut and Lebeau in~\cite{BLprep}.
We will see how it fits into Igusa's axiomatic framework of~\cite{Iax},
see section~\ref{sect5.3} above.
At the moment,
Bismut-Lebeau torsion is only defined for small positive values of~$b$.
A definition for all~$b>0$ would be nicer because
as the hypoelliptic Laplacian converges to the generator of the geodesic
flow as~$b\to\infty$, one hopes to recover some information about
the fibrewise geodesic flow from the higher torsion.

The Hermitian form~$\mathfrak{h}$ of~\eqref{eq8.17} restricts to a nondegenerate Hermitian form~$\mathfrak{h}_{b}^{H_{\pm}}$ on~$H^{\bullet}_{\pm}(E/B;F)$, so one still has characteristic forms~$\cho(H_{\pm}^{\bullet}(E/B;F),\mathfrak{h}_{b}^{H_{\pm}}) \in \Omega^{{\rm odd}}(B)$.

Bismut and Lebeau also show
that the heat operator~$e^{-\mathfrak A_{b,t,\pm}^2}$ is a smoothing operator
and of trace class.
Analytic torsion
forms~$\mathcal T_{b,\pm}(T^HE,\penalty0 g^{TM},\penalty0 g^F)\in\Omega^{\mathrm{even}}B$
can thus be defined as in section~\ref{sect2.2}.
They satisfy the following analogue of Theorem~\ref{thm1.3}.

\begin{Theorem}[Bismut and Lebeau~\cite{BLprep}]\label{thm8.5}
  For~$b>0$ sufficiently small,
  \begin{multline}\label{eq8.25}
    d\mathcal T_{b,\pm}\bigl(T^HE,g^{TM},g^F\bigr)
    =\int_{E/B}e\bigl(TM,\nabla^{TM}\bigr)\,\cho\bigl(F,g^F\bigr)\\
	-\cho\bigl(H^\bullet_\pm(E/B;F),\mathfrak h^{H_\pm}_b\bigr)\;.
  \end{multline}
\end{Theorem}

Note that~$\cho(H^\bullet_-(E/B;F))=(-1)^n\,\cho(H^\bullet_+(E/B;F))$
by~\eqref{eq8.24}.
This gives no contradiction in~\eqref{eq8.25} because for odd~$n$,
the first term on the right hand side vanishes.

It is now natural to compare~$\mathcal T_{b,\pm}$ with
the Bismut-Lott torsion~$\mathcal T$ of section~\ref{sect2.2}.
Recall that we have defined a metric~$g^H_{L^2}$ on~$H(E/B;F)$
in section~\ref{sect1.3}.
Let~$g^{H_\pm}_{L^2}$ denote the induced metric on~$H_\pm(E/B;F)$.

\begin{Theorem}[Bismut and Lebeau~\cite{BLprep}]\label{thm8.6}
  For~$b>0$ sufficiently small,
  the Hermitian form~$(\pm 1)^n\,\mathfrak h_b^{H_\pm}$
  is positive definite,
  and modulo exact forms on~$B$ one has
  \begin{multline}\label{eq8.26}
    \mathcal T_{b,\pm}\bigl(T^HE,g^{TM},g^F\bigr)
    =(\pm 1)^n\,\mathcal T\bigl(T^HE,g^{TM},g^F\bigr)\\
	-\chosl\bigl(H_\pm(E/B;F),g^H_{L^2},\mathfrak h^{H_\pm}_b\bigr)
	\pm\tr_{E/B}^*\Jnull(TM)\,\rk F\;.
  \end{multline}
\end{Theorem}

\begin{Remark}\label{rem8.7}
  If one applies the exterior derivative~$d$ on~$B$ to~\eqref{eq8.26},
  the result is compatible with Theorems~\ref{thm1.3} and~\ref{thm8.5},
  which explains the first two terms on the right hand side of~\eqref{eq8.26}.

  The last term is a Miller-Morita-Mumford class in Igusa's sense.
  It cannot be guessed from Theorems~\ref{thm1.3} and~\ref{thm8.5}.
  But if we believe that~$\mathcal T_{b,\pm}\bigl(T^HE,g^{TM},g^F\bigr)$
  is a higher torsion invariant in the sense of Definition~\ref{Def5.6},
  then it is not surprising that such a class appears here.
  On the other hand, it is surprising
  that~$\mathcal T_{b,\pm}\bigl(T^HE,g^{TM},g^F\bigr)$
  is given by the same linear combination of the classes
  in Theorem~\ref{thm5.7} as Igusa-Klein torsion
  in the following special case.
  If~$F$ is acyclic and~$E\to B$ admits a fibrewise Morse function,
  then
  \begin{equation}\label{eq8.27}
    \mathcal T_{b,-}\bigl(T^HE,g^{TM},g^F\bigr)=(-1)^n\,\tau(E/B;F)
  \end{equation}
  by comparison with Theorem~\ref{thm5.5}.
  If~$h$ has trivial stable tangent bundle~$T^sM$ in an analogous sense
  to Definition~\ref{Def4.1},
  the class $\mathcal T_{b,+}\bigl(T^HE,g^{TM},g^F\bigr)$
  equals~$\xi_h(E/B;F)^*\,\tau$ up to sign.
  Conjecturally, these equations hold even if there is no
  fibrewise Morse function.
  This coincidence indicates a relation between
  the Bismut-Lebeau analytic torsion and Igusa-Klein torsion
  that is even deeper than Theorems~\ref{thm5.5} or~\ref{thm5.7}.
\end{Remark}

\end{document}